\documentclass[12pt,reqno]{amsart}
\usepackage{amsmath,amssymb,amsfonts,amsthm}
\usepackage[left=3cm, right=3cm, top=2.8cm, bottom=2.8cm]{geometry}
\usepackage[utf8]{inputenc}
\usepackage[T1]{fontenc}
\usepackage{color}

\usepackage{color}
\usepackage{amssymb}
\usepackage{multirow}% http://ctan.org/pkg/multirow
\usepackage{graphicx}% http://ctan.org/pkg/graphicx
\usepackage{amsmath,amsthm}

\newtheorem{theorem}{Theorem}
\newtheorem{taggedtheoremx}{Theorem}
\newenvironment{taggedtheorem}[1]
 {\renewcommand\thetaggedtheoremx{#1}\taggedtheoremx}
 {\endtaggedtheoremx}
\newtheorem{lemma}[theorem]{Lemma}
\newtheorem{proposition}[theorem]{Proposition}
\newtheorem{corollary}[theorem]{Corollary}
\newtheorem{definition}[theorem]{Definition}
\newtheorem{remark}[theorem]{Remark}
\DeclareMathOperator{\esup}{ess\,sup}

\begin{document}

\title[On the fractional version of Leibniz rule]{On the fractional version of Leibniz rule}

%%%%%%%%%%%%%%%%%%%%%%%%%%%

\author[P. M. Carvalho-Neto]{Paulo M. de Carvalho-Neto}
\address[Paulo M. de Carvalho Neto]{Departamento de Matem\'atica, Centro de Ciências Físicas e Matemáticas, Universidade Federal de Santa Catarina, Florian\'{o}polis - SC, Brazil}
\email[]{paulo.carvalho@ufsc.br}
\author[R. Fehlberg Junior]{Renato Fehlberg Junior}
\address[Renato Fehlberg Junior]{Departamento de Matem\'atica, Universidade Federal do Esp\'{i}rito Santo, Vit\'{o}ria - ES, Brazil}
\email[]{fjrenato@yahoo.com.br}

%%%%%%%%%%%%%%%%%%%%%%%%%%%

\subjclass[2010]{26A33, 33B15, 35R11, 76D05}

%%%%%%%%%%%%%%%%%%%%%%%%%%%

\keywords{Caputo fractional derivative, Riemann-Liouville fractional derivative, Faedo-Galerkin method, Leibniz rule, Stokes equations.}

%%%%%%%%%%%%%%%%%%%%%%%%%%%

\begin{abstract}
%%%
This manuscript is dedicated to prove a new inequality that involves an important case of Leibniz rule regarding Riemann-Liouville and Caputo fractional derivatives of order $\alpha\in(0,1)$. In the context of partial differential equations, the aforesaid inequality allows us to address the Faedo-Galerkin method to study several kinds of partial differential equations with fractional derivative in the time variable; particularly, we apply these ideas to prove the existence and uniqueness of solution to the fractional version of the 2D Stokes equations in bounded domains.
%%%
\end{abstract}

\maketitle

\section{Introduction}

The fractional calculus is nowadays considered a prominent mathematical branch which investigate properties of derivatives and integrals of non-integer order. Historically, it emerged almost at the same time of the genesis of classical calculus and owes its origin to an inquiry raised by L'Hospital, in a letter sent to Leibniz, of whether the meaning of a derivative to an integer order could be extended to a non-integer order. For further details on the history of fractional calculus, see Ross \cite{Ro1} and Machado-Kiryakova-Mainardi \cite{MaKiMa0,MaKiMa1}.

Ever since, much has been done to settle down the cornerstones of the theory and to obtain important results; we may cite as few examples of papers recently published in journals of high impact in the mathematical society \cite{AlCaVa1,BaKoSa1,CarFe1,CarPl1,Che1,GiNa1,KeSiVeZa1,KeLiWa1,LiMu1,Pon1,ToYa1}.

In addition to the relevance of fractional calculus to mathematics as a whole, we emphasize that this theory is highly used in applied sciences. Besides the several authors that discuss the possible applications of fractional calculus to engineering, physics, biology, and others (see as general references \cite{DeAb1, LiLi1,Ma1,ZeLiLiTu1}), there is a very interesting connection between random walks, anomalous diffusion and the fractional formulation of differential equations. A precise and important survey that discuss all this connections was done by Metzler-Klafter in \cite{MeKl1}.

Nonetheless, it is important to stress that many classical and simple problems, already solved in the standard calculus theory, sometimes, can be quite complicated to be addressed using the ideas and tools of the fractional calculus. For instance, assume that $f,g:[t_0,t_1]\subset\mathbb{R}\rightarrow\mathbb{R}$ are functions with $n$th continuous derivatives, for some $n\in\mathbb{N}^*:=\{1,2,3,\ldots\}$.

\begin{itemize}
\item[(a)] The general Leibniz rule elucidates that:
\begin{equation}\label{leibniz}
\dfrac{d^n}{dt^n}\big[f(t)g(t)\big]=\displaystyle\sum_{k=0}^n{{n}\choose{k}}f^{(n-k)}(t)g^k(t),\quad \textrm{for every }t\in [t_0,t_1].
\end{equation}

\item[(b)] By recalling  Francesco Fa\`{a} di Bruno formula (see \cite{Cr1,Wa1} as good sources on this subject)
\begin{multline}\label{faa}
\dfrac{d^n}{dt^n}f\big(g(t)\big)=\\\displaystyle\sum{\dfrac{n!}{\gamma_1!\gamma_2!\ldots \gamma_n!}f^{(m)}\big(g(t)\big)\left(\dfrac{g^\prime(t)}{1!}\right)^{\gamma_1}\left(\dfrac{g^{\prime\prime}(t)}{2!}\right)^{\gamma_2}\ldots\,\,\left(\dfrac{g^{(n)}(t)}{n!}\right)^{\gamma_n}},\\ \textrm{for every }t\in [t_0,t_1],
\end{multline}
where the sum is taken over all possible combinations of nonnegative integers $\gamma_1, \gamma_2,\ldots, \gamma_n$ such that
\begin{equation*}\label{numbers}\gamma_1 + 2\gamma_2 + \ldots + n\gamma_n = n\qquad\textrm{and}\qquad \gamma_l + \gamma_2 + \ldots + \gamma_n=m.\end{equation*}
\end{itemize}

In the framework of fractional calculus we cannot expect analogous formulas to \eqref{leibniz} and \eqref{faa}, mainly because fractional derivatives have a non-local behavior, sometimes called ``memory property'', that is not compatible with these identities (see \cite{OrMa1} for more details on the concept that underlies this notion). Despite of the fact that some authors proclaim that their fractional versions of derivative satisfy these equalities (see for instance \cite{BeAkCe1,Ju1,KiIn1,Zhe1} and several others), Tarasov and Liu have already constructed sufficiently convincing arguments that invalidate such claim, as can be seen in \cite{Liu1,Ta1,Ta2,Ta3}.

On the other hand, with a remarkable argument, Podlubny in \cite{Po1}, Baleanu-Trujillo in \cite{BaTr1} and others, give a proof of the correct fractional version of \eqref{leibniz} and \eqref{faa}; these results are respectively stated bellow:

\begin{itemize}
\item[(a')] Assume that $\alpha\in(0,1)$ and $f,g:[t_0,t_1]\subset\mathbb{R}\rightarrow\mathbb{R}$ along with all its derivatives are continuous. Then
\begin{equation*}
D_{t_0,t}^\alpha\big[f(t)g(t)\big]=\displaystyle\sum_{k=0}^\infty{{\alpha}\choose{k}}f^{(k)}(t)D_{t_0,t}^{\alpha-k}g(t),\quad \textrm{for every }t\in (t_0,t_1],
\end{equation*}
and
\begin{multline*}
cD_{t_0,t}^\alpha\big[f(t)g(t)\big]=\displaystyle\sum_{k=0}^\infty{{\alpha}\choose{k}}f^{(k)}(t)D_{t_0,t}^{\alpha-k}g(t)-\dfrac{(t-t_0)^{-\alpha}f(t_0)g(t_0)}{\Gamma(1-\alpha)},\\ \textrm{for every }t\in [t_0,t_1].
\end{multline*}
\item[(b')] Like before, if we assume that $\alpha\in(0,1)$ and $f,g:[t_0,t_1]\subset\mathbb{R}\rightarrow\mathbb{R}$ along with all its derivatives are continuous, we deduce that
{\small\begin{multline*}
D_{t_0,t}^\alpha \Big[f\big(g(t)\big)\Big]=\dfrac{(t-t_0)^{-\alpha}}{\Gamma(1-\alpha)}f\big(g(t)\big)\\+\displaystyle\sum_{k=1}^\infty\binom{\alpha}{k}
\dfrac{(t-t_0)^{k-\alpha}}{\Gamma(k-\alpha+1)}\sum{\dfrac{k!}{\gamma_1!\gamma_2!\ldots \gamma_k!}f^{(m)}\big(g(t)\big)\left(\dfrac{g^\prime(t)}{1!}\right)^{\gamma_1}\left(\dfrac{g^{\prime\prime}(t)}{2!}\right)^{\gamma_2}\ldots\,\,\left(\dfrac{g^{(k)}(t)}{k!}\right)^{\gamma_k}},\\\quad \textrm{for every }t\in (t_0,t_1],
\end{multline*}}
and
{\small\begin{multline*}
cD_{t_0,t}^\alpha \Big[f\big(g(t)\big)\Big]=\dfrac{(t-t_0)^{-\alpha}}{\Gamma(1-\alpha)}\Big[f\big(g(t)\big)-f\big(g(t_0)\big)\Big]\\+\displaystyle\sum_{k=1}^\infty\binom{\alpha}{k}
\dfrac{(t-t_0)^{k-\alpha}}{\Gamma(k-\alpha+1)}\sum{\dfrac{k!}{\gamma_1!\gamma_2!\ldots \gamma_k!}f^{(m)}\big(g(t)\big)\left(\dfrac{g^\prime(t)}{1!}\right)^{\gamma_1}\left(\dfrac{g^{\prime\prime}(t)}{2!}\right)^{\gamma_2}\ldots\,\,\left(\dfrac{g^{(k)}(t)}{k!}\right)^{\gamma_k}},\\\quad \textrm{for every }t\in [t_0,t_1],
\end{multline*}}
where the sum without limits is taken over all possible combinations of nonnegative integers $\gamma_1, \gamma_2,\ldots, \gamma_k$ such that
\begin{equation*}\gamma_1 + 2\gamma_2 + \ldots + k\gamma_k = k\qquad\textrm{and}\qquad \gamma_l + \gamma_2 + \ldots + \gamma_k=m.\end{equation*}
\end{itemize}
Above, the binomial satisfies the identity
$$\binom{\alpha}{0}=1\qquad\textrm{and}\qquad\binom{\alpha}{k}=\dfrac{\alpha(\alpha-1)\ldots(\alpha-k+1)}{k!},$$
while the symbol $D_{t_0,t}^{p}$ \big($cD_{t_0,t}^{p}$\big) is used to denote the Riemann-Liouville (Caputo) fractional derivative of order $p$ at $t_0$, when $p>0$, and the Riemann-Liouville fractional integral of order $p$ at $t_0$, when $p<0$ (for more details on the definition see Section \ref{pre}).

We observe that even the particular case (and yet fundamental to study energy estimates to partial differential equations) that occurs in the standard Leibniz formula when $n=1$ and $f=g$, or in the standard chain rule when $n=1$ and $f(t)=t^2$, i.e.,
\begin{equation}\label{desifrac}
\dfrac{d}{dt}\big[g(t)\big]^2=2\left[\dfrac{d}{dt}g(t)\right]g(t),\\\quad \textrm{for every }t\in [t_0,t_1],
\end{equation}
does not have a simple formulation in the fractional calculus setting.

There are some studies in the direction of achieving some analogous version of equality \eqref{desifrac} to fractional derivatives; we may cite Shinbrot in \cite[Lemmas 5.1 and 5.2]{Sh1} as one of the first examples in the literature. However, it was Alikhanov in \cite[Lemma 1]{Al1} that manage to obtain a breakthrough on this subject.

\begin{taggedtheorem}{A}\label{alikhanov} Consider $\alpha\in(0,1)$ and assume that $f:[0,T]\rightarrow\mathbb{R}$ is an absolutely continuous function. Then
$$cD_{0,t}^\alpha \big[f(t)\big]^2\leq2\Big[cD_{0,t}^\alpha f(t)\Big]f(t),$$
for almost every $t\in [0,T]$.
\end{taggedtheorem}

In the proof of Theorem \ref{alikhanov}, the author heavily uses the restrictive fact that $f$ is absolutely continuous (cf. Definition \ref{caputo}). This kind of hypothesis is tied closely to the following result: If $\alpha\in(0,1)$ and $f:[t_0,t_1]\subset\mathbb{R}\rightarrow\mathbb{R}$ is an absolutely continuous function, it holds that (see Remark \ref{caputo1} for more details)
$$cD_{t_0,t}^\alpha f(t)=J_{t_0,t}^{1-\alpha}f^\prime(t),\quad \textrm{for almost every }t\in [t_0,t_1],$$
where $J_{t_0,t}^{1-\alpha}$ denotes the Riemann-Liouville fractional integral of order $(1-\alpha)$ at $t_0$.

We point out that Zhou-Peng stated in their paper \cite[Lemma 2.3]{ZhPe1} that a natural generalization of the above inequality should be described by the following.

\begin{taggedtheorem}{Z-P}\label{zhou} Suppose that $\alpha\in(0,1)$, $H$ is a Hilbert space and $v:[0,T]\rightarrow H$ is such that $\|v(t)\|_H^2$ is absolutely continuous. Then it holds that
$$cD_{0,t}^\alpha \big\|v(t)\big\|_H^2\leq2\,\Big(v(t),cD_{0,t}^\alpha v(t)\Big)_H,$$
for almost every $t\in [0,T]$.
\end{taggedtheorem}

Motivated by the above result, Zhou-Peng in \cite{ZhPe1} applied the Faedo-Galerkin method to study the fractional version of the Navier Stokes equations in bounded domains. However, the solutions of the reduced Feado-Galerkin equations are not (in general) absolutely continuous and therefore it is not possible to directly apply Theorem \ref{zhou} to complete the steps of the method (for more details see Remark \ref{caputo1}, Sections \ref{applications} and \ref{conclusions}).

Hence, the main objective of this work is to introduce a generalization of Theorems \ref{alikhanov} and \ref{zhou}, besides proving that this kind of result also holds for Riemann-Liouville fractional derivative of order $\alpha\in(0,1)$. We also verify that the obtained inequalities cannot be improved.

The paper is organized as follows: we begin Section \ref{pre} by introducing some special functions and its properties. Then we recall some classical notions and results from the theory of fractional calculus that are recursively used in this manuscript. In Section \ref{matrixtheory} we introduce several new ideas by using  matrix theory in order to prove Theorem \ref{polinomio}, which is a weaker version of our main results. We end this section discussing the sharpness of this theorem. Section \ref{maintheorem} contains our main results, which are Theorems \ref{finalriemann}, \ref{finalcaputo} and \ref{finalvectorial}. In Section \ref{applications} we apply the results obtained in the previous sections to prove the existence and uniqueness of solution to the fractional version of the 2D Stokes equations in bounded domains. Finally, we dedicate Section \ref{conclusions} to discuss some statements done throughout this manuscript and also to point some other applications to the theory developed here.

\section{Theoretical prerequisites}
\label{pre}
In this section we introduce the main tools used in this paper. The subjects addressed here are mainly connected with fractional calculus, which nowadays is a theory well established in the literature. There are several papers, surveys and books which can be used as references for this topic; here follows few examples of them: \cite{Car1,KiSrTr1,Po1,SaKiMa1}.

\subsection{Special functions and related results}

We start, for the sake of completeness of this paper, by introducing the gamma function, the digamma function and some properties that are fundamental to this work.

\begin{definition}\label{gammagamma} Let $\Gamma:\mathbb{C}\backslash\{0,-1,-2,\ldots\}\rightarrow \mathbb{C}$ be the analytical gamma function, which possess the following properties:
    \begin{itemize}
    \item[(a)] If $\operatorname{Re}(z)>0$, then it holds
    $$\Gamma(z)=\displaystyle\int_0^{\infty}{s^{z-1}e^{-s}}\,ds\, .$$
    Above, $\operatorname{Re}(z)$ denotes the real part of $z$.
    \item[(b)] When $n\in\mathbb{N}:=\{0,1,2,\ldots\}$,
    $$\Gamma(n+1)=n!\, .$$
    \item[(c)] If $z\in \mathbb{C}\backslash\{0,-1,-2,\ldots\}$,
    $$z\Gamma(z)=\Gamma(z+1).$$
    \end{itemize}
  \end{definition}

  Between the several results already discussed in the literature regarding the useful properties of gamma function, we begin by introducing Gautschi's inequality (see equation (7) in \cite{Gau1} for details)

\begin{theorem}[Gautschi's Inequality]\label{gautschi} If $s\in[0,1]$, it holds that
$$\left(\dfrac{1}{n+1}\right)^{1-s}\leq\dfrac{\Gamma(n+s)}{\Gamma(n+1)}\leq\left(\dfrac{1}{n}\right)^{1-s},$$
for every $n\in N^*$.
\end{theorem}

We also emphasize the following theorem proved by Alzer in \cite[Theorem 10 and its remark]{Alz1}.

\begin{theorem}[Alzer's Theorem]\label{alzer} Consider $n\in\mathbb{N}^*$ and assume that $\{a_k\}_{k=1}^n\subset\mathbb{R}$ and $\{b_k\}_{k=1}^n\subset\mathbb{R}$ satisfies:
\begin{itemize}
\item[(a)] $0\leq a_1\leq \ldots \leq a_n$ and $0\leq b_1\leq \ldots \leq b_n$.
\item[(b)] If $n\geq2$,
$$\sum_{k=1}^p{a_k}\leq \sum_{k=1}^p{b_k},$$
for every $p\in\{1,2,\ldots,n-1\}$.
\item[(c)] $\sum_{k=1}^n{a_k}= \sum_{k=1}^n{b_k}.$
\end{itemize}
In the above situation, we obtain
$$\prod_{k=1}^n\dfrac{\Gamma(x+a_k)}{\Gamma(x+b_k)}\geq1,\quad \textrm{for every }x>0.$$
\end{theorem}

Another important function introduced bellow is the so-called Digamma function.

\begin{definition} Assume that $\digamma:\mathbb{C}\backslash\{0,-1,-2,\ldots\}\rightarrow \mathbb{C}$ denotes the analytical digamma function, which is given by the relation
$$\digamma(z):=\dfrac{\Gamma^\prime(z)}{\Gamma(z)},\quad \textrm{for every }z\in\mathbb{C}\backslash\{0,-1,-2,\ldots\}.$$
For more details see \cite{AbSt1}.\end{definition}

\begin{remark}\label{remrem1}
\begin{itemize}
\item[(a)] Digamma function is the first order derivative of the logarithm of the gamma function. More precisely,
$$\dfrac{d}{dz}\Big[\log{\Gamma(z)}\Big]=\dfrac{\Gamma^\prime(z)}{\Gamma(z)}=\digamma(z),\quad \textrm{for every } \operatorname{Re}(z)>0.$$
\item[(b)] The digamma function satisfies %For every $z\in\mathbb{C}$, %with $\operatorname{Re}(z)>0$,
$$\digamma(z+1)=\dfrac{1}{z}+\digamma(z),\quad \textrm{for every }\operatorname{Re}(z)>0.$$
\end{itemize}
\end{remark}

\subsection{A small survey on fractional calculus of vectorial functions}\label{secfrac}

There are several fractional derivatives defined along the history (see \cite{MaKiMa1} for details). Nevertheless, among all these derivatives, in this work we only address the ones which nowadays are attributed to Riemann, Liouville and Caputo.

In order to establish a concise notation, throughout this section assume that $I\subset\mathbb{R}$ denotes a non-empty interval (bounded or unbounded) and $X$ a Banach space. Assume, for $1\leq p\leq \infty$, that $L^p(I;X)$ denotes the function space composed of all measurable functions (in Bochner's sense) $f:I\rightarrow X$, such that
\begin{itemize}
\item[(a)] $\|f(t)\|^p_X$ is integrable in $I$,\quad if $1\leq p<\infty$;
\item[(b)] $\esup_{t\in I}\|f(t)\|_X<\infty$,\quad if $p=\infty$.
\end{itemize}
This previous sets imbued with the respective norms
\begin{itemize}
\item[(a)] $\|f\|_{L^p(I;X)}=\displaystyle\int_{I}{\|f(s)\|^p_X}\,ds$,\quad if $1\leq p<\infty$, and
\item[(b)] $\|f\|_{L^\infty(I;X)}=\esup_{t\in I}\|f(t)\|_X$,\quad if $p=\infty$,
\end{itemize}
are Banach spaces.

The above function spaces are enough for us to define the following classical integral from the theory of fractional calculus.

\begin{definition}\label{riemannint} Let $\alpha\in(0,\infty)$, $t_0<t_1$ be fixed real numbers and $f\in L^1(t_0,t_1;X)$. The Riemann-Liouville fractional integral of order $\alpha$ at $t_0$ is defined by
\begin{equation*}J_{t_0,t}^\alpha f(t):=\dfrac{1}{\Gamma(\alpha)}\displaystyle\int_{t_0}^{t}{(t-s)^{\alpha-1}f(s)}\,ds,\quad \textrm{for almost every }t\in [t_0,t_1].\end{equation*}
\end{definition}

It is worth to pointing out that if we consider Gel'fand-Shilov function $g_\beta:\mathbb{R}\rightarrow\mathbb{R}$, for each fixed $\beta>0$, given by
$$g_\beta(t):=\left\{\begin{array}{cl} t^{\beta-1}/\Gamma(\beta),&t>0,\\
0,&t\leq0,\end{array}\right.$$
(for further information see \cite[Chapter 3]{GeSh1}) and let $f$ be equal to zero outside $[t_0,t_1]$, then we conclude that
$$J_{t_0,t}^\alpha f(t)=\big(g_{\alpha}*f\big)(t),\quad \textrm{for almost every }t\in[t_0,t_1].$$

With the purpose of discussing a broader definition to Riemann-Liouville and Caputo fractional derivatives, we consider the following function spaces:
\begin{itemize}
\item[(a)] If $n\in\mathbb{N}$, we define $C^{n}(I;X)$ as being the space composed of every $n$-times continuously differentiable function $f:I\rightarrow X$. If $I$ is a compact set, by introducing the norm
$$\|f\|_{C^{n}(I;X)}:=\left({\,\sum_{k=0}^n{\left[\sup_{t\in I}{\|f^{(k)}(t)\|_X}\right]^2}}\,\right)^{1/2},$$
it becomes a Banach space. Above $f^{(k)}(t)$ denotes the standard $k$-times derivative of $f(t)$. We also assume that $C^{\infty}(I;X)$ symbolizes the function space given by $\bigcap_{n=0}^\infty C^{n}(I;X)$.

\item[(b)] For $n\in\mathbb{N}^*$ and $1\leq p\leq \infty$, we define:
\begin{multline*}W^{n,p}(I;X):=\Big\{f\in L^p(I;X):f^{(k)}\textrm{ exists in the weak sense}\\
 \textrm{and belongs to }L^p(I;X),\textrm{ for every }k\in\{1,2,\ldots,n\}\Big\}.\end{multline*}
By considering $W^{n,p}(I;X)$ with norm
$$\|f\|_{W^{n,p}(I;X)}:=\left\{\begin{array}{ll}\left({\,\displaystyle\sum_{k=0}^n\big\|f^{(k)}\big\|^p_{L^p(I;X)}}\,\right)^{1/p},&\textrm{if }1\leq p<\infty,\\\\
\displaystyle\sum_{k=0}^n\big\|f^{(k)}\big\|_{L^\infty(I;X)}\,\,\,\,\,,&\textrm{if }p=\infty,\end{array}\right.$$
it becomes a Banach space.
\end{itemize}

From now on, for any $\alpha\in(0,\infty)$ we use the symbol $[\alpha]$ to denote the smallest integer that is greater or equal then $\alpha$.

\begin{definition}\label{riemannder} Let $\alpha\in(0,\infty)$ and $t_0<t_1$ be fixed real numbers. Assume that function $f\in L^1(t_0,t_1;X)$ and $(g_{[\alpha]-\alpha}*f)\in W^{[\alpha],1}(t_0,t_1;X)$. The Riemann-Liouville fractional derivative of order $\alpha$ at $t_0$ is given by
\begin{equation}\label{identitycaputo}
D_{t_0,t}^\alpha f(t):=\dfrac{d^{[\alpha]}}{dt^{[\alpha]}}\Big[g_{[\alpha]-\alpha}*f\Big](t),\quad\textrm{for almost every }t\in [t_0,t_1].
\end{equation}
Here $(d^{[\alpha]}/dt^{[\alpha]})$ is taken in the weak sense.
\end{definition}

\begin{remark} Observe that equality \eqref{identitycaputo} can be reinterpreted as
$$
D_{t_0,t}^\alpha f(t)=\dfrac{d^{[\alpha]}}{dt^{[\alpha]}}\left\{\dfrac{1}{\Gamma\big([\alpha]-\alpha\big)}\displaystyle\int_{t_0}^{t}{(t-s)^{[\alpha]-\alpha-1}f(s)}\,ds\right\},\quad \textrm{for almost every }t\in [t_0,t_1].
$$
\end{remark}

In spite of the fact that differential equations with Riemann-Liouville fractional derivative can be analyzed and solved (for classical surveys on this subject see {\cite{Po1,SaKiMa1}), it requires special initial conditions given in form of convolutions with the Gel'fand-Shilov function, and in general this kind of behavior lacks a clear physical interpretation.

Therefore, as an attempt to meet the requirements of physical reality, Caputo in his famous paper \cite{Ca1} reformulated the Riemann-Liouville fractional derivative, and nowadays in a more sophisticated and general version, several researchers study the Caputo fractional derivative which is completely describe below.

\begin{definition}\label{caputo} Assume that $\alpha\in(0,\infty)$ and consider $t_0<t_1$ fixed real numbers. If $f\in C^{[\alpha]-1}([t_0,t_1];X)$ and $(g_{[\alpha]-\alpha}*f)\in W^{[\alpha],1}(t_0,t_1;X)$, the Caputo fractional derivative of order $\alpha$ at $t_0$ is given by
$$cD_{t_0,t}^\alpha f(t):=D_{t_0,t}^\alpha\left[f(t)-\sum_{k=0}^{[\alpha]-1}\frac{f^{(k)}(t_0)}{k!}\big(t-t_0\big)^{k}\right],\quad \textrm{for almost every }t\in [t_0,t_1],$$
where $D_{t_0,t}^\alpha$ denotes the Riemann-Liouville fractional derivative of order $\alpha$ at $t_0$.
\end{definition}

\begin{remark}\label{caputo1} In the science areas where fractional calculus is applied, it is standard to assume that the domain of the Caputo fractional derivative of order $\alpha\in(0,1)$ is $AC([t_0,t_1],X)$, which denotes the space of every absolutely continuous function from $[t_0,t_1]$ to $X$. In this case we are able to prove, by straightforward computations (see \cite{Car1,KiSrTr1}), that
$$cD_{t_0,t}^\alpha f(t)=J_{t_0,t}^{1-\alpha}f^\prime(t),\quad \textrm{for almost every }t\in [t_0,t_1].$$

However, in order to address a more general result we avoid this kind of particularization (cf. Definition \ref{caputo}). In fact, recall that Morrey's Inequality to vector-valued functions ensures that
$$AC([t_0,t_1],X)\equiv W^{1,1}(t_0,t_1;X).$$
Nevertheless, the broader domain of Caputo fractional derivative of order $\alpha\in(0,1)$ contains several other functions besides $W^{1,1}(t_0,t_1;X)$; for instance, the Weierstrass function $w(t)$ is continuous in $[0,1]$, does not have weak derivative in $[0,1]$ and also satisfies $g_{1-\alpha}*w\in W^{1,1}(0,1;\mathbb{R})$. The proof of this fact can be found in \cite{RoWe1,RoSaLo1,We1}.
\end{remark}

There are several important properties concerning the two fractional derivatives introduced above, however in the following result we present just the ones that are referred in this manuscript.

\begin{proposition}\label{properties} Let $\alpha,\gamma\in(0,1)$, $t_0<t_1$ be fixed real numbers and consider functions $f,\tilde{f}\in L^1(t_0,t_1;X)$ and $h,\tilde{h}\in C^0([t_0,t_1];X)$. Then the following statements are true.
    \begin{itemize}
    \item[(a)] For Riemann-Liouville fractional integral:
    \begin{itemize}
    \item[(i)] given $\lambda,\mu\in\mathbb{R}$, it holds that
    $$J_{t_0,t}^{{\alpha}}\left[\lambda f+\mu\tilde{f}\right](t)=
    \lambda\, \left[J_{t_0,t}^{{\alpha}}f(t)\right]+\mu\, \left[J_{t_0,t}^{{\alpha}}\tilde{f}(t)\right],\quad \textrm{for almost every }t\in [t_0,t_1];$$
    \item [(ii)]  $J_{t_0,t}^{{\alpha}}\left[J_{t_0,t}^{{\gamma}}f(t)\right]=J_{t_0,t}^{\alpha+\gamma}f(t),\quad \textrm{for almost every }t\in [t_0,t_1].$\vspace*{0.2cm}
    \end{itemize}
    \item[(b)] For Riemann-Liouville fractional derivative:
    \begin{itemize}
    \item [(iii)] if functions $g_{1-{\alpha}}*f$, $g_{1-{\alpha}}*\tilde{f}$ belongs to $W^{1,1}(t_0,t_1;X)$ and $\lambda,\mu\in\mathbb{R}$, then
    $$D_{t_0,t}^{{\alpha}}\left[\lambda f+\mu\tilde{f}\right](t)=
    \lambda\left[D_{t_0,t}^{{\alpha}}f(t)\right]+\mu\left[D_{t_0,t}^{{\alpha}}\tilde{f}(t)\right],\quad \textrm{for almost every }t\in [t_0,t_1];$$
    \item [(iv)] $D_{t_0,t}^{{\alpha}}\left[J_{t_0,t}^{{\alpha}}f(t)\right]=f(t),\quad \textrm{for almost every }t\in [t_0,t_1];$\vspace*{0.2cm}
    \item [(v)] If $g_{1-{\alpha}}*f\in W^{1,1}(t_0,t_1;X)$, then
    $$J_{t_0,t}^{{\alpha}}\left[D_{t_0,t}^{{\alpha}}f(t)\right]=f(t)-\dfrac{1}{\Gamma({\alpha})}(t-t_0)^{{\alpha}-1}\left\{J^{1-{\alpha}}_{t_0,s}f(s)\right\}\Big|_{s=t_0},\quad \textrm{for almost every }t\in [t_0,t_1].$$
    Moreover, if there exists $\varphi\in L^1(t_0,t_1;X)$ such that $f(t)=J_{t_0,t}^{{\alpha}}\varphi(t)$, for almost every $t\in[t_0,t_1]$, or $f\in C^0([t_0,t_1];X)$, then
    $$J_{t_0,t}^{{\alpha}}D_{t_0,t}^{{\alpha}}f(t)=f(t),\quad \textrm{for almost every }t\in [t_0,t_1].$$
    \end{itemize}
    \item[(c)] For Caputo fractional derivative:
    \begin{itemize}
    \item [(vi)] if functions $g_{1-{\alpha}}*h$, $g_{1-{\alpha}}*\tilde{h}$ belongs to $W^{1,1}(t_0,t_1;X)$ and $\lambda,\mu\in\mathbb{R}$, then
    $$cD_{t_0,t}^{{\alpha}}\left[\lambda h+\mu\tilde{h}\right](t)=
    \lambda\left[cD_{t_0,t}^{{\alpha}}h(t)\right]+\mu\left[cD_{t_0,t}^{{\alpha}}\tilde{h}(t)\right],\quad \textrm{for almost every }t\in [t_0,t_1];$$
    \item [(vii)] $cD_{t_0,t}^{{\alpha}}\left[J_{t_0,t}^{{\alpha}}h(t)\right]=h(t),\quad \textrm{for almost every }t\in [t_0,t_1];$
    \item [(viii)] If $g_{1-{\alpha}}*h\in W^{1,1}(t_0,t_1;X)$, then
    $$J_{t_0,t}^{{\alpha}}\left[cD_{t_0,t}^{{\alpha}}h(t)\right]=h(t)-h(t_0),\quad \textrm{for almost every }t\in [t_0,t_1].$$
    \end{itemize}
    \end{itemize}
\end{proposition}

\begin{proof} We refer to \cite{Car1,KiSrTr1,Po1}.
\end{proof}

There are some examples, which are used forward in this manuscript, that should be evidenced here. For this purpose, assume that $\alpha\in(0,1)$, $k\in\mathbb{N}$, $t_0<t_1$ are fixed real numbers and that $\vartheta:\mathbb{R}\rightarrow\mathbb{R}$ is the function
    $$\vartheta(t):=(t-t_0)^{k}.$$
\begin{itemize}
\item[(a)] Using the definition of the Riemann-Liouville fractional derivative, we obtain
\begin{equation}\label{continha1}D_{t_0,t}^\alpha \vartheta(t)=\left\{\begin{array}{cll}\dfrac{\Gamma(k+1)}{\Gamma(k+1-\alpha)}(t-t_0)^{k-\alpha},& \textrm{for every }t\in [t_0,t_1],&\textrm{if } k>0,\vspace*{0.2cm}\\ \dfrac{1}{\Gamma(1-\alpha)}(t-t_0)^{-\alpha},& \textrm{for every }t\in (t_0,t_1],&\textrm{if } k=0.\end{array}\right.\end{equation}

\item[(b)] Like before, using the definition of Caputo fractional derivative we achieve
\begin{equation}\label{continha2}cD_{t_0,t}^\alpha \vartheta(t)=\left\{\begin{array}{cll}\dfrac{\Gamma(k+1)}{\Gamma(k+1-\alpha)}(t-t_0)^{k-\alpha},& \textrm{for every }t\in [t_0,t_1],&\textrm{if } k>0,\vspace*{0.2cm}\\0,& \textrm{for every }t\in [t_0,t_1],&\textrm{if } k=0.\end{array}\right.\end{equation}
\end{itemize}

\begin{remark}\label{caputo1supernovo} By analyzing the definitions of the fractional derivatives addressed here, observe that for functions $f\in C^{0}([t_0,t_1];X)$ such that $g_{1-\alpha}*f\in W^{1,1}(t_0,t_1;X)$ it always hold
$$cD_{t_0,t}^{\alpha}f(t)=D_{t_0,t}^{\alpha}f(t)-\dfrac{(t-t_0)^{-\alpha}}{\Gamma(1-\alpha)}f(t_0)\quad \textrm{for almost every }t\in (t_0,t_1].$$
This identity is recurrently used in this manuscript.
\end{remark}

\section{Matrix analysis and first results}
\label{matrixtheory}

Taking into account the considerations presented so far, we now state Theorem \ref{polinomio} which is our first main result.

\begin{theorem}\label{polinomio} Consider $\alpha\in(0,1)$, $t_0\in\mathbb{R}$ and $P:\mathbb{R}\rightarrow\mathbb{R}$ a polynomial function with real coefficients. Then we have
\begin{equation}\label{des1111}
D_{t_0,t}^\alpha\big[P(t)\big]^2\leq2\Big[D_{t_0,t}^\alpha P(t)\Big]P(t),\quad \textrm{for every }t>t_0,
\end{equation}
and
\begin{equation}\label{des2222}
 cD_{t_0,t}^\alpha\big[P(t)\big]^2\leq2\Big[cD_{t_0,t}^\alpha P(t)\Big]P(t),\quad \textrm{for every }t>t_0.
\end{equation}
\end{theorem}

Since the goal of this section is to prove Theorem \ref{polinomio}, in what follows we introduce all the results and technicalities that are needed to completely address it.

\subsection{Auxiliary results}

The proof of the aforementioned theorem is motivated by several matrix results. This is why we begin by introducing the following notion.

\begin{definition} For $m\in\mathbb{N}^*$, a symmetric matrix $M\in M^{m}(\mathbb{R})$ is said to be positive definite when
$$\big(x, Mx\big)>0,\quad \textrm{for every }x\in \mathbb{R}^m\setminus\{0\}.$$
Above, the symbol $(\cdot\,,\cdot)$ denotes the standard inner product of $\mathbb{R}^m$.\end{definition}

Now fix $n\in\mathbb{N}$, with $n\geq2$, and consider the partitioned matrix $B\in M^{n+1}(\mathbb{R})$, given by
\begin{equation}\label{eq10}
 B=\left[
  \begin{array}{cc}
    d & e^{T} \\
    e & A \\
  \end{array}
\right],
\end{equation}
where $A\in M^{n}(\mathbb{R})$ is a symmetric matrix, $e=(e_1,e_2,\ldots,e_n)\in \mathbb{R}^n$, which is viewed as a column matrix, $e^{T}$ a row matrix given by the transpose of $e$ and $d\in\mathbb{R}$.

By taking into account Schur complement theory, we may address the following result regarding the partitioned matrix  $B$ introduced above (see \cite[Theorem 7.7.6]{HoJo1} as a classical reference for this result).

\begin{theorem}\label{matrix01} Assume that $B\in M^{n+1}(\mathbb{R})$ is given by \eqref{eq10}, with $d\not=0$, and define matrix
$$E=\left[\begin{array}{cccc}\dfrac{e_1e_1}{d}&\dfrac{e_1e_2}{d}&\ldots&\dfrac{e_1e_n}{d}\vspace{0.3cm}\\
\dfrac{e_2e_1}{d}&\dfrac{e_2e_2}{d}&\ldots&\dfrac{e_2e_n}{d}\vspace{0.3cm}\\
\vdots&\vdots&\ddots&\vdots\vspace{0.3cm}\\
\dfrac{e_ne_1}{d}&\dfrac{e_ne_2}{d}&\ldots&\dfrac{e_ne_n}{d}\end{array}\right].$$
Then $B$ is positive definite if, and only if, $d>0$ and $D:=A-E$ is a positive definite matrix.
\end{theorem}

In the following results our objective is to ensure enough conditions to verify that $D$, defined in the above theorem, is indeed a positive definite matrix. Thus, from this moment on, we use recursively the notations introduced by \eqref{eq10} and Theorem \ref{matrix01}.

\begin{theorem}\label{matrix02} If $d\in\mathbb{R}\setminus\{0\}$, $e\in\mathbb{R}^n\setminus\{0\}$ and $(e,Ae)>0$, then %
$$\mathbb{R}^n=\operatorname{span}(e)\oplus\operatorname{span}(Ae)^\perp=\operatorname{span}(e)\oplus \operatorname{Ker}E.$$
\end{theorem}

\begin{proof}
  Since $e\not=0$ and $(e,Ae)>0$, it holds that $Ae\neq 0$ and therefore $\operatorname{span}(Ae)^\perp$ is a vector subspace of $\mathbb{R}^n$ of dimension $n-1$. Let $\{u_2,\ldots, u_n\}$ be a base of $\operatorname{span}(Ae)^\perp$. A standard computation shows that $\{e, u_2,\ldots, u_n\}$ is a base of $\mathbb{R}^n$.

  To prove that $\mathbb{R}^n=\operatorname{span}(e)\oplus\operatorname{Ker}E$, first notice the following identity
  \begin{equation}\label{ex}
    Ex=\frac{(e,x)}{d}e.
  \end{equation}
  Now consider $\{w_1,\ldots, w_{n-1}\}$ a base of $\operatorname{span}(e)^\perp$. Then it holds that $Ew_i=0$, for $i\in \{1,2,\ldots, n-1\}$,
  what ensures that $\operatorname{span}(e)^\perp\subset\operatorname{Ker}E$. However, we known that
  $$n-1=\operatorname{dim}(\operatorname{span}(e)^\perp)\leq\operatorname{dim}(\operatorname{Ker}(E)),$$
  and since $E\neq 0$, it follows that $\operatorname{dim}(\operatorname{Ker}E)\leq n-1$. Thus, the result follows.
\end{proof}

In view of Theorem \ref{matrix02}, for each $x\in\mathbb{R}^n$ there is $\alpha\in\mathbb{R}$ such that $x=\alpha e+u$ for some $u\in\operatorname{span}(Ae)^\perp=\operatorname{Ker}E$. Therefore, we introduce the following result.

\begin{lemma}\label{lem1}
  Assume the same hypotheses of Theorem \ref{matrix02}. Then
  $$(x,Dx)=\alpha^2\left[(e,Ae)-\|e\|^4/d\right]+(u,Au),\quad \textrm{for every }x=(\alpha e+u)\in\mathbb{R}^n.$$
  \end{lemma}

\begin{proof}
  Since $E$ is symmetric and $u\in\operatorname{Ker}E$, we see by \eqref{ex} that $(x,Ex)=\alpha^2\|e\|^4/d$. On the other hand, since $(e,Au)=(u,Ae)$ and $u\in\operatorname{span}(Ae)^\perp$, we deduce the identity
  $$(x,Ax)=\alpha^2(e,Ae)+(u,Au).$$
  The result now follows.
\end{proof}

Using the above lemma we can state the following theorem.

\begin{theorem}\label{theo1}
  If in addition to the hypotheses of Theorem \ref{matrix02}, we assume that $A$ is a positive definite matrix that satisfies
  $$(e,Ae)\geq\|e\|^4/d,$$
  we conclude that $D$ is a positive definite matrix.
\end{theorem}

\begin{proof} This is a straightforward application of Lemma \ref{lem1}.
\end{proof}

We end this subsection addressing a corollary that is used to prove inequality \eqref{des1111}.

\begin{corollary}\label{auxteorema}
  Let $B$ be a matrix as in \eqref{eq10} such that $A=(a_{ij})_{n\times n}$ is a positive definite matrix and $d>0$. If $e_i= d$ and $a_{ij}\geqslant d$, then $B$ is a positive definite matrix.
\end{corollary}

\begin{proof}
 Since $A$ is a positive definite matrix that satisfies
\begin{equation*}\label{sn3}
  (e,Ae)-\frac{\|e\|^4}{d}\geq\sum_{i,j=1}^{n}{d^3}-\frac{\|e\|^4}{d}=n^2d^3-\frac{\|e\|^4}{d}=0,
\end{equation*}
Theorem \ref{theo1} ensures that $D:=A-E$ is a definite positive matrix. Thus, Theorem \ref{matrix01} completes this proof.
\end{proof}

\subsection{Initial considerations and the proof of Theorem \ref{polinomio}}

Assume that $\alpha\in(0,1)$, $t_0\in\mathbb{R}$ and $P_n(t):=\sum_{k=0}^{n}{b_kt^k}$ is a polynomial function with real coefficients in the variable $t$.
Observe that we can rewrite $P_n(t)=\sum_{k=0}^{n}{a_k(t-t_0)^k}$, where $a_0,a_1,\ldots,a_n$ are given by
$$a_k=\sum_{i=k}^nb_i\,{{i}\choose{k}}\,t_0^{i-k},\quad \textrm{for every }k\in\{0,1,\ldots,n\}.$$
Thus, together with item $(iii)$ of Proposition \ref{properties} and identity \eqref{continha1}, we obtain
\begin{multline}\label{des111}
  {D}^\alpha_{t_0,t}[P_n(t)^2]= {D}^\alpha_{t_0,t}\left[\sum_{i,j=0}^{n}{a_ia_j(t-t_0)^{i+j}}\right]
=\sum_{i,j=0}^{n}{\left[\frac{a_ia_j\Gamma(i+j+1)}{\Gamma(i+j+1-\alpha)}\right](t-t_0)^{i+j-\alpha}},\\ \textrm{for every }t>t_0,
\end{multline}
and
\begin{equation*}
  \left[{D}^\alpha_{t_0,t}P_n(t)\right]P_n(t)=\sum_{i,j=0}^{n}{\left[\frac{a_ia_j\Gamma(i+1)}{\Gamma(i+1-\alpha)}\right](t-t_0)^{i+j-\alpha}},\quad \textrm{for every }t>t_0,
\end{equation*}
which implies
\begin{multline}\label{des133}
  2\left[{D}^\alpha_{t_0,t}P_n(t)\right]P_n(t)=\sum_{i,j=0}^{n}{a_ia_j\left[\frac{\Gamma(i+1)}{\Gamma(i+1-\alpha)}+\frac{\Gamma(j+1)}{\Gamma(j+1-\alpha)}\right](t-t_0)^{i+j-\alpha}},\\ \textrm{for every }t>t_0.
\end{multline}

Let us assume for a moment that inequality \eqref{des1111} holds true. Then, if we apply equations \eqref{des111} and \eqref{des133} to it, we would obtain

\begin{equation}\label{des14}
  0\leq \sum_{i,j=0}^{n}{a_ia_j\big[\psi(i,j)\big](t-t_0)^{i+j-\alpha}},\quad \textrm{for every }t>t_0,
\end{equation}
where $\psi:\mathbb{N}\times\mathbb{N}\rightarrow\mathbb{R}$ is the symmetric function defined by
\begin{equation}\label{gammapsi11}
  \psi(i,j)=\frac{\Gamma(i+1)}{\Gamma(i+1-\alpha)}+\frac{\Gamma(j+1)}{\Gamma(j+1-\alpha)}-\frac{\Gamma(i+j+1)}{\Gamma(i+j+1-\alpha)}.
\end{equation}

Now we can reinterpret inequality \eqref{des14} as
\begin{equation}\label{des1515}
  0\leq \Big(v_a(t),\mathcal{B}v_a(t)\Big),
\end{equation}
where $v_a(t):=\big(a_0,a_1(t-t_0),\ldots,a_n(t-t_0)^n\big)$ and $\mathcal{B}=\big(\psi(i,j)\big)$ is a symmetric matrix of order $n+1$, with $i,j\in\{0,\ldots,n\}$.

By repeating the above procedure to the Caputo fractional derivative (recall item (vi) of Proposition \ref{properties} and equation \eqref{continha2}), we achieve
\begin{multline}\label{nova01}
  \qquad c{D}^\alpha_{t_0,t}[P_n(t)^2]=2\sum_{i=1}^{n}{\left[\frac{a_0a_i\Gamma(i+1)}{\Gamma(i+1-\alpha)}\right](t-t_0)^{i-\alpha}}\\
  +\sum_{i,j=1}^{n}{\left[\frac{a_ia_j\Gamma(i+j+1)}{\Gamma(i+j+1-\alpha)}\right](t-t_0)^{i+j-\alpha}},\quad \textrm{for every }t>t_0,\qquad
\end{multline}
and
\begin{multline}\label{nova02}
  2\left[c{D}^\alpha_{t_0,t}P_n(t)\right]P_n(t)=2\sum_{i=1}^{n}{\frac{a_0a_i\Gamma(i+1)}{\Gamma(i+1-\alpha)}(t-t_0)^{i-\alpha}}\\
  +\sum_{i,j=1}^{n}{a_ia_j\left[\frac{\Gamma(i+1)}{\Gamma(i+1-\alpha)}+\frac{\Gamma(j+1)}{\Gamma(j+1-\alpha)}\right](t-t_0)^{i+j-\alpha}},\quad \textrm{for every }t>t_0.
\end{multline}

Like before, if we assume that inequality \eqref{des2222} holds true and replace equations \eqref{nova01} and \eqref{nova02} in it, we shall deduce
\begin{equation}\label{des2121}
  0\leq \sum_{i,j=1}^{n}{a_ia_j\big[\psi(i,j)\big](t-t_0)^{i+j}}=\Big(u_a(t),\mathcal{A}u_a(t)\Big),
\end{equation}
where $u_a(t)=\big(a_1(t-t_0),\ldots,a_n(t-t_0)^n\big)$, $\mathcal{A}=\big(\psi(i,j)\big)$, with $i,j\in\{1,\ldots,n\}$, is a matrix of order $n$ and $\psi(i,j)$ is the function defined in \eqref{gammapsi11}.

By comparing \eqref{des1515} and \eqref{des2121} we observe the relation
\begin{equation}\label{eq2}
 \mathcal{B}=\left[
  \begin{array}{cc}
    \delta & \varepsilon^{T} \\
    \varepsilon & \mathcal{A} \\
  \end{array}
\right]
\end{equation}
where $\delta=\psi(0,0)$ and $\varepsilon^{T}=(\varepsilon_1,\varepsilon_2,\ldots,\varepsilon_n)$, with $\varepsilon_i=\psi(i,0)$, for $i\in\{1,2,\ldots,n\}$.

The discussion developed above leads us to state the following crucial result.

\begin{proposition}\label{ultiulti} Let $\mathcal{A}$ and $\mathcal{B}$ be the matrices considered in \eqref{eq2}.
\begin{itemize}
\item[(a)] Matrix $\mathcal{A}$ is positive definite if, and only if, \eqref{des2222} holds true.
\item[(b)] Matrix $\mathcal{B}$ is positive definite if, and only if, \eqref{des1111} holds true.
\end{itemize}
\end{proposition}

\begin{proof} The proof of item $(a)$ follows from the relation between \eqref{des2222} and \eqref{des2121}, while the proof of item $(b)$ from the relation between \eqref{des1111} and \eqref{des1515}.
\end{proof}

Proposition \ref{ultiulti} together with inequalities \eqref{des1515} and \eqref{des2121} indicates that we should first understand some properties of function $\psi$ before proceed with the proof of Theorem \ref{polinomio}. Hence, we dedicated the remainder of this subsection to study these properties.

\begin{proposition}\label{theo2}
  Let $\psi:\mathbb{N}\times\mathbb{N}\rightarrow\mathbb{R}$ be the function defined in \eqref{gammapsi11} and assume that $i\geq 0$ and $j>0$. Then, for every $k>0$,
  $$\textrm{(a)}\,\,\,\psi(i+k,j)>\psi(i,j)\quad\quad\quad\quad\textrm{and}\quad\quad\quad\quad\textrm{(b)}\,\,\,\psi(i+k,0)=\psi(i,0)=\psi(0,0)>0.$$

\end{proposition}

\begin{proof}
  Here we only address $k=1$, since the general situation is a consequence of this case.

  $(a)$ Let $i\geq0$ and $j>0$. By recalling gamma function properties, we get
  \begin{multline*}\label{Gamma1}
  % \nonumber to remove numbering (before each equation)
    \psi(i+1,j)=\frac{\Gamma(i+2)}{\Gamma(i+2-\alpha)}+\frac{\Gamma(j+1)}{\Gamma(j+1-\alpha)}-\frac{\Gamma(i+j+2)}{\Gamma(i+j+2-\alpha)} \\
     =\left(\frac{i+1}{i+1-\alpha}\right)\frac{\Gamma(i+1)}{\Gamma(i+1-\alpha)}+\frac{\Gamma(j+1)}{\Gamma(j+1-\alpha)}\\-\left(\frac{i+j+1}{i+j+1-\alpha}\right)\frac{\Gamma(i+j+1)}{\Gamma(i+j+1-\alpha)}.
  \end{multline*}

  Hence, we conclude that
  \begin{equation}\label{Gamma2222}
    \psi(i+1,j)-\psi(i,j)=\alpha\left(\frac{\Gamma(i+1)}{\Gamma(i+2-\alpha)}-\frac{\Gamma(i+j+1)}{\Gamma(i+j+2-\alpha)}\right).
  \end{equation}

  Since $\alpha\in(0,1)$, equation \eqref{Gamma2222} is strictly positive if, and only if,
  \begin{equation}\label{Gamma3333}
    \frac{\Gamma(i+1)\Gamma(i+j+2-\alpha)}{\Gamma(i+2-\alpha)\Gamma(i+j+1)}>1.
  \end{equation}
  However, inequality \eqref{Gamma3333} holds for $i\geq 0$ and $j>0$ by Theorem \ref{alzer}; just define $x=i+1$, $a_1=0$, $a_2=j+1-\alpha$, $b_1=1-\alpha$ and $b_2=j$.

  $(b)$ Let $i\geq0$. Notice that
  $$\psi(i,0)=\frac{\Gamma(i+1)}{\Gamma(i+1-\alpha)}+\frac{1}{\Gamma(1-\alpha)}-\frac{\Gamma(i+1)}{\Gamma(i+1-\alpha)}=\frac{1}{\Gamma(1-\alpha)}=\psi(0,0)>0.$$
\end{proof}

Now we introduce two important results. %The first one discuss a property of function $\psi:\mathbb{N}^2\rightarrow\mathbb{R}$, which was defined in \eqref{gammapsi11}, while the second one give conditions to ensure that inequality \eqref{des1111} holds.

\begin{theorem}\label{ultimolema} Let $\mathcal{A}$ and $\mathcal{B}$ be the matrices given in \eqref{eq2}. If $\mathcal{A}$ is a positive definite matrix, then $\mathcal{B}$ is a positive definite matrix.
\end{theorem}

\begin{proof} Since $\mathcal{A}$ is a positive definite matrix and Proposition \ref{theo2} ensures that
\[
  \mathcal{B}=\left[\begin{array}{cccc}
    1/{\Gamma(1-\alpha)} & \cdots & 1/{\Gamma(1-\alpha)}\vspace*{0.2cm}\\
    \raisebox{-2pt}{\vdots} & \multicolumn{3}{c}{\multirow{3}{*}{\raisebox{4mm}{\scalebox{2}{$\mathcal{A}$}}}} \vspace*{0.2cm}\\
    1/{\Gamma(1-\alpha)} & & &
  \end{array}\right],
\]
with $1/{\Gamma(1-\alpha)}>0$, we conclude by Corollary \ref{auxteorema} that $\mathcal{B}$ is a positive definite matrix.
\end{proof}

\begin{corollary}\label{ultiultiultiulti} If inequality \eqref{des2222} holds true, then inequality \eqref{des1111} also holds true.
\end{corollary}
\begin{proof} It is a consequence of Proposition \ref{ultiulti} and Theorem \ref{ultimolema}.
\end{proof}

\begin{remark} If we consider that Alikhanov in \cite{Al1} already proved inequality \eqref{des2222} for  absolutely continuous functions, then Corollary \ref{ultiultiultiulti} is enough to ensure the validity of inequality \eqref{des1111}. Nonetheless, in what follows we introduce a new proof for inequality \eqref{des2222} using the theory already established in this section. \end{remark}

To prove that $\mathcal{A}$ is a positive definite matrix, we use an induction argument on its order. For this purpose, for each $n\geq2$, we reinterpret matrix $\mathcal{A}$ as being given by
$$\mathcal{A}_n:=\left[\begin{array}{cccc}
    \psi(1,1) & \psi(1,2) & \cdots & \psi(1,n) \\
    \psi(2,1)  & \psi(2,2) & \cdots & \psi(2,n)\\
    \vdots & \vdots & \ddots &\vdots\\
    \psi(n,1)  & \psi(n,2) & \cdots & \psi(n,n)
  \end{array}\right].
$$

In the next results we change the issue of verifying that matrix $\mathcal{A}_n$ is positive definite by an easier problem. In fact, the following proposition begins this approach.

\begin{proposition}\label{propprop1} Consider $n\in\mathbb{N}$. Matrix $\mathcal{A}_n$ is positive definite if, and only if, matrix $\widetilde{\mathcal{A}}_n:=\big(\psi(n+1-i,n+1-j)\big)_{n\times n}$ is positive definite.
\end{proposition}
\begin{proof} This follows directly from the fact that $\mathcal{A}_n$ is a symmetric matrix.
\end{proof}

Next corollary uses almost the same steps already implemented in Theorem \ref{ultimolema}, however here matrix $\widetilde{\mathcal{A}}_{n}$ does not have a constant first line and column.

\begin{corollary}\label{finalimplementado} For $n\geq2$, if $\widetilde{\mathcal{A}}_{n}$ is a positive definite matrix and
\begin{equation}\label{psi111}(h_n,\widetilde{\mathcal{A}}_nh_n)\geq\dfrac{\|h_n\|^4}{\psi(n+1,n+1)},\end{equation}
then $\widetilde{\mathcal{A}}_{n+1}$ is a positive definite matrix.
\end{corollary}

\begin{proof} Observe that
$$\widetilde{\mathcal{A}}_{n+1}=\left[
  \begin{array}{cc}
    \psi(n+1,n+1) & h_n^{T} \\\\
    h_n & \widetilde{\mathcal{A}}_n \\
  \end{array}
\right],$$
where $h_n^{T}=\big(\psi(n+1,n),\ldots,\psi(n+1,1)\big)$.

Now assume that $\widetilde{\mathcal{A}}_n$ is a positive definite matrix. Since \eqref{psi111} holds true and Proposition \ref{theo2} ensures that the value of $\psi(n+1,n+1)$ is positive, Theorems \ref{matrix01} and \ref{theo1} guarantees that $\widetilde{\mathcal{A}}_{n+1}$ is a positive definite matrix.
\end{proof}

At this point in our study, it is necessary to explain our modus operandi. In fact, Corollary \ref{finalimplementado} guarantees that inequality $\eqref{psi111}$ is fundamental to prove that $\widetilde{\mathcal{A}}_n$ is a positive definite matrix, for any $n\geq2$. Hence, we proceed in a sequence of reductions of this inequality.

We start pointing out that $\eqref{psi111}$ is equivalent to
\begin{equation}\label{psi1}
  \sum_{i,j=1}^{n}\bigg\{\psi(i,n+1)\psi(j,n+1)\Big[\psi(i,j)\psi(n+1,n+1)-\psi(i,n+1)\psi(j,n+1)\Big]\bigg\}\geq0.
\end{equation}

In order to verify that inequality \eqref{psi1} holds, it is enough to prove that
\begin{equation}\label{wn2}
  \frac{\psi(n+1,n+1)}{\psi(i,n+1)}\geq\frac{\psi(j,n+1)}{\psi(i,j)},\quad \textrm{for every }i,j\in\{1,\ldots,n\} \textrm{ and } n\in \mathbb{N}^*.
\end{equation}

On the other hand, for fixed values $n\in \mathbb{N}^*$ and $i\in\{1,\ldots,n\}$, if we manage to obtain the inequality
\begin{equation}\label{hi2}
  \frac{\psi(j+1,n+1)}{\psi(i,j+1)}\geq\frac{\psi(j,n+1)}{\psi(i,j)},
\end{equation}
for every $j\in\{1,2,\ldots,n\}$, then \eqref{wn2} becomes a consequence of it.

Now, in order to verify inequality \eqref{hi2}, we first prove the following result.

\begin{proposition}\label{lem2}
  Consider function $\phi:\mathbb{N}\times\mathbb{N}\rightarrow\mathbb{R}$ given by
  \begin{equation*}\label{hA}
    \phi(i,j):=\frac{\Gamma(i+1)}{\Gamma(i+2-\alpha)}-\frac{\Gamma(i+j+1)}{\Gamma(i+j+2-\alpha)}.
  \end{equation*}
  Then the following relations are true:
  \begin{itemize}
    \item[(a)] $\psi(i+1,j)=\psi(i,j)+\alpha \phi(i,j)$,\,\, for every\,\, $i,j\in\mathbb{N}$;
    \item[(b)] $\phi(i,j)>0$,\,\, for\,\, $i\in\mathbb{N}$ and $j\in\mathbb{N}^*$;
    \item[(c)] $\phi(i,j+1)=\phi(i,j)+\phi(i+j,1)$,\,\, for every\,\, $i,j\in\mathbb{N}$;
    \item[(d)]$\phi(i,j)=\sum_{k=0}^{j-1}{\phi(i+k,1)}$,\,\, for every\,\, $i,j\in\{1,2,\ldots,n\}$,\,\, where $n\in\mathbb{N}^*$;
    \item[(e)]$\psi(i,j)= \psi(1,1)-\alpha \phi(0,1)+\alpha \left[\sum_{k=0}^{i-1}{\sum_{l=0}^{j-1}{\phi(k+l,1)}}\right]$, for every \\$i,j\in\{1,2,\ldots,n\}$, where $n\in\mathbb{N}^*$.
  \end{itemize}
\end{proposition}

\begin{proof}
Items $(a)$ and $(b)$ are direct consequences of Proposition \ref{theo2}. Item $(c)$ is just a straightforward computation and item $(d)$ is a consequence of item $(c)$.

Now we focus in the proof of item $(e)$. Assume that $i\geq2$ and $j\geq2$ (the case $i=1$ or $j=1$ is obtained from analogous arguments). Then, item $(a)$ ensures that
$$\psi(i,j) = \psi(1,j)+\alpha \left[\sum_{k=1}^{i-1}{\phi(k,j)}\right]= \psi(1,1)+\alpha \left[\sum_{l=1}^{j-1}{\phi(l,1)}\right]+\alpha \left[\sum_{k=1}^{i-1}{\phi(k,j)}\right],$$
since $\psi(1,j)=\psi(j,1)$.

But item $(d)$ allows us to conclude
$$\psi(i,j) = \psi(1,1)+\alpha \left[\sum_{l=1}^{j-1}{\phi(l,1)}\right] +\alpha \left[\sum_{k=1}^{i-1}{\sum_{l=0}^{j-1}{\phi(k+l,1)}}\right].$$

By adding and subtracting $\alpha\phi(0,1)$ from the right side of the above equality, and by reorganizing the sums, we achieve item $(e)$.
 \end{proof}

Since $\psi(i,j)$ is symmetric and considering item $(a)$ of Proposition \ref{lem2}, inequality \eqref{hi2} can be reinterpreted as
\begin{equation*}
  \frac{\psi(j,n+1)+\alpha \phi(j,n+1)}{\psi(j,i)+\alpha\phi(j,i)}\geq\frac{\psi(j,n+1)}{\psi(i,j)},
\end{equation*}
which is equivalent to
\begin{equation}\label{hA3}
  \frac{\psi(i,j)}{\phi(j,i)}\geq\frac{\psi(j,n+1)}{\phi(j,n+1)},
\end{equation}
for every $i,j\in\{1,2,\ldots,n\}$ and $n\in \mathbb{N}^*$.

However, inequality \eqref{hA3} can be derived directly from %, if we manage to prove that $m_j$ is a non increasing function; just observe that
\begin{equation}\label{m2}
  \frac{\psi(i,j)}{ \phi(j,i)}\geq\frac{\psi(j,i+1)}{ \phi(j,i+1)},
\end{equation}
if we prove it for every $i,j\geq1$. Finally, observe that items $(a)$ and $(c)$ of Proposition \ref{lem2} and a straightforward computation ensures the equivalence of \eqref{m2} and
\begin{equation}\label{m5}
  \psi(i,j)\phi(j+i,1)\geq\alpha \phi(i,j)\phi(j,i),
\end{equation}
for every $i,j\geq1$.

Our final equivalent inequality, which is obtained by applying items $(d)$ and $(e)$ of Proposition \ref{lem2} in inequality \eqref{m5}, is given by
\begin{multline}\label{m6}
   \left[\psi(1,1)-\alpha \phi(0,1)\right]\phi(i+j,1)\\+\alpha \left[\sum_{k=0}^{i-1}{\sum_{l=0}^{j-1}{\phi(k+l,1)\phi(i+j,1)-\phi(i+l,1)\phi(j+k,1)}}\right]\geq0.
\end{multline}

 Now, item $(b)$ of Proposition \ref{lem2} allows us to conclude that
 $$\left[\psi(1,1)-\alpha \phi(0,1)\right]\phi(i+j,1)>0\Longleftrightarrow \psi(1,1)-\alpha \phi(0,1)>0.$$
 But since
 $$\psi(1,1)-\alpha \phi(0,1)>0\Longleftrightarrow (2-\alpha)(1-\alpha)>0,$$
 we conclude that the first term of \eqref{m6} is positive.

 Thus, to prove that inequality \eqref{m6} holds, and therefore that inequality \eqref{psi1} is true, we present our last proposition.

\begin{proposition}\label{lem33}
  Given $i,j\in\mathbb{N}$ and $k,l$ such that $0\leq k\leq i-1$ and $0\leq l\leq j-1$, then
\begin{equation}\label{A0}
  {\phi(k+l,1)\phi(i+j,1)-\phi(i+l,1)\phi(j+k,1)}>0.
\end{equation}
\end{proposition}

\begin{proof}  By definition $\phi(i,1)=(1-\alpha)\Gamma(i+1)/\Gamma(i+3-\alpha)>0$ and therefore
\begin{equation}\label{A1}
  \phi(k+l,1)\phi(i+j,1)=(1-\alpha)^2\left[\frac{\Gamma(k+l+1)\Gamma(i+j+1)}{\Gamma(k+l+3-\alpha)\Gamma(i+j+3-\alpha)}\right]
\end{equation}
and
\begin{equation}\label{A2}
  \phi(i+l,1)\phi(j+k,1)=(1-\alpha)^2\left[\frac{\Gamma(i+l+1)\Gamma(j+k+1)}{\Gamma(i+l+3-\alpha)\Gamma(j+k+3-\alpha)}\right].
\end{equation}
Hence, by considering \eqref{A1} and \eqref{A2}, inequality \eqref{A0} can be rewritten as
\begin{equation*}
 \frac{\Gamma(k+l+1)\Gamma(i+j+1)}{\Gamma(k+l+3-\alpha)\Gamma(i+j+3-\alpha)}>\frac{\Gamma(i+l+1)\Gamma(j+k+1)}{\Gamma(i+l+3-\alpha)\Gamma(j+k+3-\alpha)},
\end{equation*}
which is equivalent to
\begin{equation}\label{A4}
 \frac{\Gamma(k+l+1)\Gamma(i+j+1)}{\Gamma(i+l+1)\Gamma(j+k+1)}>\frac{\Gamma(k+l+3-\alpha)\Gamma(i+j+3-\alpha)}{\Gamma(i+l+3-\alpha)\Gamma(j+k+3-\alpha)}.
\end{equation}

Now, consider the infinitely differentiable function $f:[0,\infty)\rightarrow\mathbb{R}$ given by
\begin{equation*}
 f(s):=\frac{\Gamma(x+s)\Gamma(y+s)}{\Gamma(z+s)\Gamma(w+s)},
\end{equation*}
for fixed numbers $x,y,z,w\in(0,\infty)$.

By considering item (a) of Remark \ref{remrem1}, we conclude that the derivative of $f(s)$ is given by
\begin{equation*}
 \frac{\Gamma(x+s)\Gamma(y+s)\Gamma(z+s)\Gamma(w+s)\left[\digamma(x+s)+\digamma(y+s)-\digamma(z+s)-\digamma(w+s)\right]}{\Gamma(z+s)^2\Gamma(w+s)^2},
\end{equation*}
where $\digamma(z)$ denotes de digamma function. Then, $f'(s)<0$ if, and only if,
$$\digamma(x+s)+\digamma(y+s)-\digamma(z+s)-\digamma(w+s)<0.$$

Firstly observe that item (b) of Remark \ref{remrem1} ensures the equality
$$\digamma(p+s)=\left[\dfrac{1}{p-1+s}\right]+\digamma(p-1+s)=\ldots=\sum_{q=0}^{p-1}{\left[\dfrac{1}{q+s}\right]}+\digamma(s),$$
for each $p\in\mathbb{N}^*$. Thus, by choosing $x=k+l+1$, $y=i+j+1$, $z=i+l+1$ and $w=j+k+1$, the above equality allows us to conclude that
\begin{multline}\label{ultimoteste}
  \digamma(x+s)+\digamma(y+s)-\digamma(z+s)-\digamma(w+s)\\=\sum_{q=0}^{k+l}{\left[\frac{1}{q+s}\right]}+\sum_{q=0}^{i+j}{\left[\frac{1}{q+s}\right]}-\sum_{q=0}^{i+l}{\left[\frac{1}{q+s}\right]}-\sum_{q=0}^{j+k}{\left[\frac{1}{q+s}\right]}.
\end{multline}

Now, since by hypotheses $k\leq i-1$ and $l\leq j-1$, we rewrite \eqref{ultimoteste} as
\begin{equation*}
  \digamma(x+s)+\digamma(y+s)-\digamma(z+s)-\digamma(w+s)=\sum_{q=i+l+1}^{i+j}{\left[\frac{1}{q+s}\right]}-\sum_{q=k+l+1}^{k+j}{\left[\frac{1}{q+s}\right]},
\end{equation*}
and therefore,
\begin{align*}
  \digamma(x+s)+\digamma(y+s)-\digamma(z+s)-\digamma(w+s)&=\sum_{q=l+1}^{j}{\left[\frac{1}{q+i+s}\right]}-\sum_{q=l+1}^{j}{\left[\frac{1}{q+k+s}\right]}\\
  &=\sum_{q=l+1}^{j}{\left[\frac{1}{q+i+s}-\frac{1}{q+k+s}\right]}<0.
\end{align*}

The above computations ensure that $f'(s)<0$, that is, $f$ is a decreasing function in $s$. Then, since $\alpha\in(0,1)$, we have $f(0)>f(2-\alpha)$ which is exactly inequality \eqref{A4}. This completes the proof.\end{proof}

Finally we are able to present the whole proof of Theorem \ref{polinomio}.

\begin{proof}[Proof of Theorem~\ref{polinomio}] We begin by proving inequality \eqref{des2222}. Recall that item (a) of Proposition \ref{ultiulti} ensures that it is enough for us to prove that $\mathcal{A}$, defined in \eqref{des2121}, is a positive definite matrix. Now, observe that Proposition \ref{propprop1} ensures that $\mathcal{A}$ is a positive definite matrix if, and only if, $\widetilde{\mathcal{A}}_n$ is a positive definite matrix, for every $n\geq2$.

From now on, we proceed with an induction argument. At first notice that
$$\widetilde{\mathcal{A}}_{2}=\left[
  \begin{array}{cc}
    \psi(2,2) & \psi(1,2) \\\\
    \psi(2,1) & \psi(1,1) \\
  \end{array}
\right],$$
and the determinant of its leading principal minors are
$$\psi(2,2)=4\dfrac{(6-\alpha)(1-\alpha)}{\Gamma(5-\alpha)}>0$$
and
$$\psi(1,1)\psi(2,2)-\psi(1,2)^2=
\dfrac{\alpha(1-\alpha)^2(2-\alpha)(6-\alpha)}{\Gamma(4-\alpha)\Gamma(5-\alpha)}>0.$$
Thus, $\widetilde{\mathcal{A}}_2$ is a positive definite matrix.

Now, assume that $\widetilde{\mathcal{A}}_n$ is a positive definite matrix and let us show that $\widetilde{\mathcal{A}}_{n+1}$ is positive definite matrix. Recall that inequality \eqref{psi111} derives from Proposition \ref{lem33} (as it was discussed throughout this subsection), and therefore Corollary \ref{finalimplementado} ensures that $\widetilde{\mathcal{A}}_{n+1}$ is positive definite matrix. This finishes the induction argument, completing in this way the proof of inequality \eqref{des2222}.

Lastly, \eqref{des1111} derives from Corollary \ref{ultiultiultiulti}.\end{proof}

\begin{remark}\label{esperoqueultimo} \begin{itemize}\item[(a)] At first we emphasize that the proof developed in this section, actually, ensures that inequality \eqref{des1111} holds strictly when $P(t)$ is not the null polynomial function.
\item[(b)] On the other hand, inequality \eqref{des2222} does not own this property; recall that the Caputo fractional derivative of a constant function is indeed zero (see the identity \eqref{continha2}), therefore \eqref{des2222} is strict just when we consider non constant polynomial functions.
\item[(c)] By continuity, inequality \eqref{des2222} holds for every $t\geq t_0$. The same argument does not hold for \eqref{des1111}. This is due to the fact that Riemann-Liouville fractional derivative sometimes is not defined at $t_0$; and example of this is the constant polynomial function.
\end{itemize}
\end{remark}

We end this section by presenting two results that ensure the sharpness of Theorem \ref{polinomio}; in fact, they are obtained as a consequence of Theorem \ref{gautschi}.

\begin{lemma}\label{gautschiaux} Define the values $\varphi_k:=\Gamma(k+1)/\Gamma(k+1-\alpha)$, for each $k\in\mathbb{N}$.
\begin{itemize}
\item[(a)] The sequence $\{\varphi_k\}_{k=1}^\infty$ is increasing and $\lim_{k\rightarrow\infty}{\varphi_k}=\infty.$

\item[(b)] The sequence $\{{\varphi_{2k}}/{\varphi_k}\}_{k=1}^\infty$ is increasing and converges to $2^\alpha$.\end{itemize}
\end{lemma}
\begin{proof} $(a)$ At first, observe that
$$\varphi_{k+1}>\varphi_k\Longleftrightarrow \dfrac{k+1}{k+1-\alpha}>1,$$
which holds for each $k\in\mathbb{N}$. Now Theorem \ref{gautschi} implies that $\lim_{k\rightarrow\infty}{\varphi_k}=\infty.$

$(b)$ Note that
    $$\dfrac{\varphi_{2(k+1)}}{\varphi_{k+1}}>\dfrac{\varphi_{2k}}{\varphi_k}
    \quad\Longleftrightarrow\quad\dfrac{\Gamma(2k+3)\Gamma(k+1)\Gamma(k+2-\alpha)\Gamma(2k+1-\alpha)}{\Gamma(2k+1)\Gamma(k+2)\Gamma(k+1-\alpha)\Gamma(2k+3-\alpha)}>1.$$
    By classical properties of Gamma function (see item $(c)$ of Definition \ref{gammagamma}), the latest inequality is equivalent to
    $$\dfrac{2(2k+1)(k+1-\alpha)}{(2k+1-\alpha)(2k+2-\alpha)}>1.$$
    Since last inequality holds true, we conclude the first part of the proof. The second part is a consequence of Theorem \ref{gautschi}.
\end{proof}

\begin{theorem} Assume that $\lambda\in\mathbb{R}\setminus\{2\}$.
\begin{itemize}
\item[(a)] Then, there exists a polynomial function with real coefficients $P_\lambda(t)$ satisfying
\begin{equation}\label{finalfinaltest00}
D_{t_0,t}^\alpha\big[P_\lambda(t)\big]^2>\lambda\Big[D_{t_0,t}^\alpha P_\lambda(t)\Big]P_\lambda(t),\quad \textrm{for some }t>t_0.
\end{equation}
\item[(b)] Also, there exists a polynomial function with real coefficients $Q_\lambda(t)$ satisfying
\begin{equation}\label{finalfinaltest}
cD_{t_0,t}^\alpha\big[Q_\lambda(t)\big]^2>\lambda\Big[cD_{t_0,t}^\alpha Q_\lambda(t)\Big]Q_\lambda(t),\quad \textrm{for some }t>t_0.
\end{equation}
\end{itemize}
\end{theorem}

\begin{proof} For simplicity we assume that $t_0=0$. Nonetheless, it worths to point out that the general situation follows the same ideas used bellow.
\begin{itemize}
\item[(b)] To obtain the proof of this item, we split it in three cases.
\begin{itemize}
\item[(i)] If $\lambda\leq2/(2-\alpha)$, choose $Q_\lambda(t):=t+1$ and observe that \eqref{continha2}, Proposition \ref{properties} and Gamma function properties ensure the equivalence between \eqref{finalfinaltest} and the inequality
    \begin{equation}\label{finalfinaltest2}
    t\big[2-\lambda(2-\alpha)\big]+(2-\alpha)(2-\lambda)>0,
    \end{equation}
    which holds for every $t>0$.
\item[(ii)] For $2/(2-\alpha)<\lambda<2$ and $Q_\lambda(t):=t+1$, identity \eqref{finalfinaltest2} holds for every %
    $$0<t<\dfrac{(2-\alpha)(2-\lambda)}{\lambda(2-\alpha)-2}.$$
\item[(iii)] Finally, if $\lambda>2$ and we consider $Q_\lambda(t):=t-1$, then \eqref{finalfinaltest} is equivalent to
    \begin{equation*}
    t\big[2-\lambda(2-\alpha)\big]+(2-\alpha)(\lambda-2)>0,
    \end{equation*}
    which holds for every
    $$0<t<\dfrac{(2-\alpha)(\lambda-2)}{\lambda(2-\alpha)-2}.$$
\end{itemize}

\item[(a)]  We also split this proof in three cases. Notice that in this situation we need another approach.

\begin{itemize}
\item[(i)] Assume that $\lambda<2^\alpha$ is fixed and consider $P_k(t):=t^k+1$, where $k\in\mathbb{N}^*$. Observe that the existence of $\tilde{t}>0$ such that \eqref{finalfinaltest00} holds, is equivalent to
\begin{equation}\label{classicalcalc}f_k(\tilde{t}):=\underbrace{\left[\varphi_{2k}-\lambda\varphi_k\right]}_{A_k}\big(\,\tilde{t}\,\big)^{2k}+
\underbrace{\big[(2-\lambda)\varphi_k-\lambda\varphi_0\big]}_{B_k}\big(\,\tilde{t}\,\big)^k+\underbrace{(1-\lambda)\varphi_0}_{C}>0,\end{equation}
where $\varphi_k:=\Gamma(k+1)/\Gamma(k+1-\alpha)$, for each $k\in\mathbb{N}^*$.

Since Lemma \ref{gautschiaux} ensures that $\{{\varphi_{2k}}/{\varphi_k}\}$ is an increasing sequence that converges to $2^\alpha$, there exists $k_0=k_0(\lambda)\in\mathbb{N}$ such that ${\varphi_{2k}}/{\varphi_k}>\lambda$, for every $k\geq k_0$. Hence, we conclude that $A_k>0$, for every $k\geq k_0$, and therefore that
$$\lim_{t\rightarrow\infty}{f_{k_0}(t)}=\infty.$$
This implies that there exists $t_0> 0$ such that $f_{k_0}(t)>0$, for all $t\in [t_0,\infty)$.

\item[(ii)] For a fixed value $2^\alpha\leq\lambda<2$, consider $P_k(t):=t^k+1$. Like before, inequality \eqref{finalfinaltest00} is equivalent to \eqref{classicalcalc}, but now $A_k<0$, for every $k\in\mathbb{N}^*$. On the other hand, observe that
$$B_k>0\quad\Longleftrightarrow \quad\dfrac{2\varphi_k}{\varphi_k+\varphi_0}>\lambda.$$

Since Lemma \ref{gautschiaux} ensures that $\{2\varphi_k/(\varphi_k+\varphi_0)\}_{k=1}^\infty$ is an increasing sequence that satisfies
$$\dfrac{2\varphi_k}{\varphi_k+\varphi_0}\rightarrow 2,$$
when $k\rightarrow \infty$, there should exist a natural number $k_1=k_1(\lambda)$ such that
$$\dfrac{2\varphi_k}{\varphi_k+\varphi_0}>\lambda,$$
for any $k\geq k_1$.

By taking $\tilde{t}_k:=\sqrt[k]{-B_{k}/(2A_{k})}>0$, for any $k\geq k_1$, we compute that
%
%$$f_{n}(\tilde{t}_n)>0\quad\Longleftrightarrow \quad \underbrace{B_{n}^2-4A_{n}C}_{D_{n}}>0.$$
$$f_{k}(\tilde{t}_k)>0\quad\Longleftrightarrow \quad B_{k}^2-4A_{k}C>0.$$

However, $B_{k}^2-4A_{k}C>0$ can be reinterpreted as
$$(2-\lambda)^2\varphi_{k}^2-2\lambda^2\varphi_0\varphi_{k}+\lambda^2\varphi_0-4(1-\lambda)\varphi_0\varphi_{2k}>0,$$
which in turn, is equivalent to
$$(2-\lambda)^2\varphi_{k}-2\lambda^2\varphi_0+\lambda^2\left[\dfrac{\varphi_0}{\varphi_{k}}\right]-4(1-\lambda)\varphi_0\left[\dfrac{\varphi_{2k}}{\varphi_{k}}\right]>0.$$
By applying Lemma \ref{gautschiaux} we conclude that this inequality holds for sufficiently large values of $k$; this ensures that the proof of this item is completed.

\item[(iii)] For the case $\lambda>2$, consider $P_k(t):=t^k-1$ and observe that inequality \eqref{finalfinaltest00} is equivalent to
\begin{equation*}\underbrace{\left[\varphi_{2k}-\lambda\varphi_k\right]}_{A_k}t^{2k}-
\underbrace{\big[(2-\lambda)\varphi_k-\lambda\varphi_0\big]}_{B_k}t^k+\underbrace{(1-\lambda)\varphi_0}_{C}>0,\end{equation*}
for some $t>0$. By the conclusions obtained in the previous items, we already know that in this case $A_k<0$ and $-B_k>0$, for any $k\in\mathbb{N}$. But then, by the same arguments implemented in item $(ii)$ we conclude that this item holds true for sufficiently large values of $k$.
\end{itemize}
\end{itemize}
\end{proof}

\section{Main results}
\label{maintheorem}

Last section discussed the proof of a sharp inequality that relates Leibniz rule with the fractional derivatives (Caputo and Riemann-Liouville) of polynomial functions. However, these inequalities are valid in a much broader aspect. Observe that throughout all this section we assume that $\alpha\in(0,1)$ and $t_0<t_1$ are fixed real numbers.

\subsection{A first generalization of Theorem \ref{polinomio}}

 We begin with an auxiliary proposition that is recurrently used in this manuscript.

\begin{proposition}\label{samkokilbas} Assume that $X$ is a Banach space and $1\leq p\leq\infty$. Then, there exists a constant $K=K(p,\alpha,t_0,t_1)>0$, such that
$${\left\|J_{t_0,t}^\alpha f\right\|_{L^p(t_0,t_1;X)}}\leq K\left\|f\right\|_{L^p(t_0,t_1;X)}\,\,,$$
for any function $f\in L^p(t_0,t_1;X)$.
\end{proposition}

\begin{proof} The proof is an adaptation (to Banach spaces) of Samko-Kilbas-Marichev \cite[page 48]{SaKiMa1}.
\end{proof}

Next we present a first improvement of Theorem \ref{polinomio}, which is not our most robust result, however, plays an important role in our forward proofs.

\begin{theorem}\label{c1functions} If $f$ belongs to $C^1([t_0,t_1];\mathbb{R})$, then
\begin{equation}\label{ultultrl}
D_{t_0,t}^\alpha\big[f(t)\big]^2\leq2\Big[D_{t_0,t}^\alpha f(t)\Big]f(t),\quad \textrm{for every }t\in (t_0,t_1],
\end{equation}
and
\begin{equation}\label{ultultcpt}
cD_{t_0,t}^\alpha\big[f(t)\big]^2\leq2\Big[cD_{t_0,t}^\alpha f(t)\Big]f(t),\quad \textrm{for every }t\in [t_0,t_1].
\end{equation}
\end{theorem}

\begin{proof} Since the polynomial functions are dense in $C^1([t_0,t_1];\mathbb{R})$ (with its standard topology), for every $f\in C^1([t_0,t_1];\mathbb{R})$ there exists a sequence of polynomial functions $\{P_k(t)\}_{k=1}^\infty$ such that
\begin{equation}\label{ini01imp}\sup_{s\in[t_0,t_1]}{\big|P_k(s)-f(s)\big|}+\sup_{s\in[t_0,t_1]}{\left|P^\prime_k(s)-f^\prime(s)\right|}\rightarrow 0,\end{equation}
when $k\rightarrow\infty$.

But then Remarks \ref{caputo1} and \ref{caputo1supernovo} ensure that
\begin{align*}\left|D_{t_0,t}^\alpha P_k(t)-D_{t_0,t}^\alpha f(t)\right|&\leq\left|J_{t_0,t}^{1-\alpha} P^\prime_k(t)-J_{t_0,t}^{1-\alpha} f^\prime(t)\right|+(t-t_0)^{-\alpha}\left|P_k(t_0)-f(t_0)\right|/\Gamma(1-\alpha),
\end{align*}
for every $t\in(t_0,t_1]$. Hence, by applying Proposition \ref{samkokilbas} (when $p=\infty$) we obtain
\begin{equation}\label{ini02imp}\left|D_{t_0,t}^\alpha P_k(t)-D_{t_0,t}^\alpha f(t)\right|\leq K_1\left[(t-t_0)^{-\alpha}\sup_{s\in[t_0,t_1]}{\big|P_k(s)-f(s)\big|}+\sup_{s\in[t_0,t_1]}{\left|P^\prime_k(s)-f^\prime(s)\right|}\right],\end{equation}
for every $t\in (t_0,t_1]$.

Following the same idea presented above, we deduce
\begin{align*}\left|D_{t_0,t}^\alpha [P_k(t)]^2-D_{t_0,t}^\alpha [f(t)]^2\right|\leq&\left|D_{t_0,t}^\alpha \left\{[P_k(t)]^2-[P_k(t_0)]^2\right\}-D_{t_0,t}^\alpha \left\{[f(t)]^2-[f(t_0)]^2\right\}\right|\\&+\left|D_{t_0,t}^\alpha [P_k(t_0)]^2-D_{t_0,t}^\alpha [f(t_0)]^2\right|\\
=&\,\,2\left|J_{t_0,t}^{1-\alpha} [P^\prime_k(t)P_k(t)]-J_{t_0,t}^{1-\alpha} [f^\prime(t)f(t)]\right|\\&+(t-t_0)^{-\alpha}\left|[P_k(t_0)]^2-[f(t_0)]^2\right|/\Gamma(1-\alpha),
\end{align*}
for every $t\in (t_0,t_1]$, and therefore
\begin{multline}\label{ini015imp}\left|D_{t_0,t}^\alpha [P_k(t)]^2-D_{t_0,t}^\alpha [f(t)]^2\right|\leq K_2\left[(t-t_0)^{-\alpha}\sup_{s\in[t_0,t_1]}{\big|[P_k(s)]^2-[f(s)]^2\big|}\right.
\\\left.+\sup_{s\in[t_0,t_1]}{\left|P^\prime_k(s)-f^\prime(s)\right|}\sup_{s\in[t_0,t_1]}{\big|P_k(s)\big|}+
\sup_{s\in[t_0,t_1]}{\big|f^\prime(s)\big|}\sup_{s\in[t_0,t_1]}{\left|P_k(s)-f(s)\right|}\right],\end{multline}
for every $t\in (t_0,t_1]$.

The proof now follows when we apply \eqref{ini01imp}, \eqref{ini02imp} and \eqref{ini015imp} in the first inequality of Theorem \ref{polinomio}. The second inequality proposed by this theorem has an analogous proof (however in this situation without the restriction $t>t_0$; see also item $(c)$ of Remark \ref{esperoqueultimo}) and therefore is omitted.
\end{proof}

\begin{remark}
\begin{itemize}
\item[(a)] Observe that for any function in $C^1([t_0,t_1];\mathbb{R})$ we can compute Riemann-Liouville and Caputo fractional derivatives of order $\alpha$; this is mainly due to Remarks \ref{caputo1} and \ref{caputo1supernovo}.
\item[(b)] In general, if $f$ belongs to $C^{1}\big([t_0,t_1]);\mathbb{R}\big)$, there is no way to ensure that $D_{t_0,t}^\alpha f$ belongs to $C^{0}\big([t_0,t_1]);\mathbb{R}\big)$. In fact, $D_{t_0,t}^\alpha f\in C^{0}\big((t_0,t_1]);\mathbb{R}\big)$; take as an example the constant function. This is why \eqref{ultultrl} does not necessarily holds at $t_0$.
\end{itemize}
\end{remark}

\subsection{An improvement of inequality (\ref{ultultrl})}

Our objective at this point is to improve the first inequality of Theorem \ref{c1functions}. To this end, we first recall a classical functional analysis result.

\begin{proposition}\label{novamolli1} Consider $\varrho\in C^\infty(\mathbb{R};\mathbb{R})$ with compact support and $\int_{-\infty}^\infty\varrho(s)ds=1$. For each $\varepsilon>0$, define function $\varrho_\varepsilon(t):=\varepsilon^{-1}\varrho(t\varepsilon^{-1})$. Assume that $X$ is a Banach space.
\begin{itemize}
\item[(a)] If $1\leq p<\infty$ and $f\in L^p\big(\mathbb{R};X)$, define $f_\varepsilon=\varrho_\varepsilon*f$. Then $f_\epsilon\in C_c^\infty(\mathbb{R};X)$ and
$$\big\|f_\varepsilon-f\big\|_{L^p(\mathbb{R};X)}\rightarrow0,$$
as $\varepsilon\rightarrow 0$.
\item[(b)] If $f\in L^\infty\big(\mathbb{R};X)$ and $f$ is continuous in an open set $U\subset\mathbb{R}$, consider as before $f_\varepsilon=\varrho_\varepsilon*f$. Hence, $f_\epsilon\in C_c^\infty(\mathbb{R};X)$ and for any compact $K\subset U$ it holds that
$$\big\|f_\varepsilon-f\big\|_{C^0(K;X)}\rightarrow0,$$
as $\varepsilon\rightarrow 0$.
\end{itemize}
\end{proposition}

\begin{remark} We emphasize that $*$ is used to denote the standard convolution. More specifically, if $h:\mathbb{R}\rightarrow \mathbb{R}$ is a measurable function and $g:\mathbb{R}\rightarrow X$ is measurable in Bochner's sense, then
$$h*g(t):=\int_{-\infty}^\infty{h(t-s)g(s)}\,ds,$$
for every $t\in\mathbb{R}$, assuming that the integral in question exists.
\end{remark}

Now we address an identity related to the Riemann-Liouville fractional integral.

\begin{proposition}\label{novamolli2} If $X$ is a Banach space, $p>1/(1-\alpha)$ and $f\in L^p(t_0,t_1;X)$, it holds that
$$J_{t_0,s}^\alpha f(s)\big|_{s=t_0}=0.$$
\end{proposition}
\begin{proof} Just observe that Holder's inequality ensures
\begin{multline*}\big\|J_{t_0,t}^\alpha f(t)\big\|_X=\left\|\dfrac{1}{\Gamma(\alpha)}\int_{t_0}^t{(t-s)^{-\alpha}f(s)}\,ds\right\|_X
\\\leq\dfrac{1}{\Gamma(\alpha)}\left(\int_{t_0}^t{(t-s)^{-p^*\alpha}}\,ds\right)^{1/p^*}\left(\int_{t_0}^t{\|f(s)\|_X^p}\,ds\right)^{1/p},\end{multline*}
where $p^*=p/(p-1)$. Since $p>1/(1-\alpha)$, then $1-p^*\alpha>0$ and therefore we deduce
$$\big\|J_{t_0,t}^\alpha f(t)\big\|_X\leq\left(\dfrac{(t-t_0)^{(1-\alpha p^*)/p^*}}{\big[(1-\alpha p^*)\big]^{1/p^*}\Gamma(\alpha)}\right)\|f\|_{L^p(t_0,t_1;X)},$$
which ensures the desired result.
\end{proof}

\begin{remark} Observe that the above proof also ensure that $J_{t_0,t}^\alpha f(t)$ is continuous in $[t_0,t_1]$.
\end{remark}

Bearing in mind the results discussed so far, we finally present the main theorem of this subsection.

\begin{theorem}\label{finalriemann} Let $p>1/(1-\alpha)$ and $p\geq2$. If $f\in L^p(t_0,t_1;\mathbb{R})$, $g_{1-\alpha}*f\in W^{1,2}(t_0,t_1;\mathbb{R})$ and $g_{1-\alpha}*f^2\in W^{1,1}(t_0,t_1;\mathbb{R})$, we obtain
\begin{equation*}
D_{t_0,t}^\alpha\big[f(t)\big]^2\leq2\Big[D_{t_0,t}^\alpha f(t)\Big]f(t),\quad \textrm{for almost every }t\in [t_0,t_1].
\end{equation*}
\end{theorem}

\begin{proof} At first assume that $t_0=0$. Consider $F\in L^p(\mathbb{R};\mathbb{R})$, with
$$F(t)=\left\{\begin{array}{ll}f(t),&\textrm{for almost every }t\in[0,t_1],\\
0,&\textrm{otherwise}.
\end{array}\right.$$
Choose $\varrho\in C^\infty(\mathbb{R};\mathbb{R})$ with compact support contained in $(0,t_1)$ satisfying 
$$\int_{-\infty}^\infty\varrho(s)ds=1.$$
In this way, Proposition \ref{novamolli1} ensures
\begin{equation}\label{ccc1}\left\{\begin{array}{l}F_\varepsilon\in C^\infty([0,t_1];\mathbb{R})\,\,\textrm{and}\vspace{0.2cm}\\
F_\varepsilon\rightarrow f,\textrm{ when }\varepsilon\rightarrow0,\,\,\textrm{in the topology of }L^p(0,t_1;\mathbb{R}).\end{array}\right.\end{equation}

Now, observe that for any $t\in[0,t_1]$
$$D_{0,t}^\alpha F_\varepsilon(t)=\dfrac{d}{dt}\left\{\dfrac{1}{\Gamma(1-\alpha)}\int_0^t{(t-s)^{-\alpha}\left[\int_0^s{\varrho_\varepsilon(s-\tau)f(\tau)}\,d\tau\right]}\,ds\right\},$$
by the commutative and associative properties of convolutions and Leibniz integral rule, we achieve
$$D_{0,t}^\alpha F_\varepsilon(t)=\dfrac{d}{dt}\left[\int_0^t{\varrho_\varepsilon(t-s) J_{0,s}^{1-\alpha}f(s)}\,ds\right]=\int_0^t{\left[\frac{d}{dt}\varrho_\varepsilon(t-s)\right]J_{0,s}^{1-\alpha}f(s)}\,ds.$$

Since Proposition \ref{novamolli2} ensures that $J_{0,s}^{1-\alpha}f(s)|_{s=0}=0$, we finally obtain
\begin{multline}\label{star}D_{0,t}^\alpha F_\varepsilon(t)=-\int_0^t{\frac{d}{ds}\left[\varrho_\varepsilon(t-s)\right]J_{0,s}^{1-\alpha}f(s)}\,ds=
\int_0^t{\varrho_\varepsilon(t-s)\frac{d}{ds}\left[J_{0,s}^{1-\alpha}f(s)\right]}\,ds\\=
\int_0^t{\varrho_\varepsilon(t-s)D_{0,s}^\alpha f(s)}\,ds=\varrho_\varepsilon*G(t)=G_\varepsilon(t),\end{multline}
where $G\in L^2(\mathbb{R};\mathbb{R})$ is given by
$$G(t)=\left\{\begin{array}{ll}D_{0,t}^\alpha f(t),&\textrm{for almost every }t\in(0,t_1],\\
0,&\textrm{otherwise}.
\end{array}\right.$$

Hence, Proposition \ref{novamolli1} and identity \eqref{star} ensure that
\begin{equation}\label{ccc2}\left\{\begin{array}{l}D_{0,t}^\alpha F_\varepsilon=G_\varepsilon\in C^\infty\big([0,t_1];\mathbb{R}\big)\,\,\textrm{and}\vspace{0.2cm}\\
D_{0,t}^\alpha F_\varepsilon\rightarrow D_{0,t}^\alpha f,\textrm{ when }\varepsilon\rightarrow0,\,\,\textrm{in the topology of }L^2(0,t_1;\mathbb{R}).\end{array}\right.\end{equation}

In order to give continuity to this proof, just recall that Theorem \ref{c1functions} guarantees
\begin{equation*}
D_{0,t}^\alpha\big[F_\varepsilon(t)\big]^2\leq2\Big[D_{0,t}^\alpha F_\varepsilon(t)\Big]F_\varepsilon(t),\quad \textrm{for every }t\in (0,t_1].
\end{equation*}
Hence,  if $\phi\in C^\infty\big([0,t_1];\mathbb{R}\big)$ has compact support contained in $(0,t_1)$ and $\phi(t)\geq0$, for every $t\in(0,t_1)$, we deduce the inequality
$$0\leq\int_{0}^{t_1}{\bigg\{2\Big[D_{0,s}^\alpha F_\varepsilon(s)\Big]F_\varepsilon(s)-D_{0,s}^\alpha\big[F_\varepsilon(s)\big]^2\bigg\}}\phi(s)\,ds,$$
which is equivalent to
\begin{equation}\label{cccmainlimite}0\leq\int_{0}^{t_1}{2\Big[D_{0,s}^\alpha F_\varepsilon(s)\Big]F_\varepsilon(s)}\phi(s)\,ds+\int_{0}^{t_1}\bigg\{{J_{0,s}^{1-\alpha}\big[F_\varepsilon(s)\big]^2}\bigg\}\phi^\prime(s)\,ds.\end{equation}

By applying \eqref{ccc1}, \eqref{ccc2} and Proposition \ref{samkokilbas} in \eqref{cccmainlimite}, we obtain
\begin{equation*}0\leq\int_{0}^{t_1}{2\Big[D_{0,s}^\alpha f(s)\Big]f(s)}\phi(s)\,ds+\int_{0}^{t_1}\bigg\{{J_{0,s}^{1-\alpha}\big[f(s)\big]^2}\bigg\}\phi^\prime(s)\,ds.\end{equation*}

Since $g_{1-\alpha}*f^2\in W^{1,1}(0,t_1;\mathbb{R})$, we achieve
\begin{equation*}0\leq\int_{0}^{t_1}{\bigg\{2\Big[D_{0,s}^\alpha f(s)\Big]f(s)-D_{0,s}^{\alpha}\big[f(s)\big]^2\bigg\}}\phi(s)\,ds.\end{equation*}
%
%for every $\phi\in C^\infty\big((t_0,t_1);\mathbb{R}\big)$ that has compact support contained in $(t_0,t_1)$ and $\phi(t)\geq0$, for every $t\in(t_0,t_1)$.

By an argument coming from Du Bois-Reymond's lemma we complete the proof of the case $t_0=0$.

When $t_0\not=0$, just observe that for $t\in[t_0,t_1]$
\begin{multline*}D_{t_0,t}^\alpha\big[f(t)\big]^2=\dfrac{d}{dt}\left\{\dfrac{1}{\Gamma(1-\alpha)}\int_{t_0}^{t}{(t-s)^{1-\alpha}[f(s)]^2}\,ds\right\}\\
=\dfrac{d}{dt}\left\{\dfrac{1}{\Gamma(1-\alpha)}\int_{0}^{t-t_0}{(t-t_0-s)^{1-\alpha}[f(s+t_0)]^2}\,ds\right\}=D_{0,s}^\alpha\big[H(s)\big]^2\big|_{s=t-t_0}\,\,,\end{multline*}
with $H(t)=f(t+t_0)$. But then the first part of this proof ensures that
$$D_{t_0,t}^\alpha\big[f(t)\big]^2=D_{0,s}^\alpha\big[H(s)\big]^2\big|_{s=t-t_0}\leq2 \left[D_{0,s}^\alpha H(s)|_{s=t-t_0}\right]H(t-t_0)=2\left[D_{t_0,t}^\alpha f(t)\right]f(t),$$
as we wanted. This completes the proof.
\end{proof}

\begin{remark} \begin{itemize}
\item[(a)] The regularity assumptions in Theorem \ref{finalriemann} are not artificial. In fact, they are a mix between:
\begin{itemize}
\item[(i)] The weaker necessary conditions to ensure the existence of $D_{t_0,t}[f(t)]^2$ and $D_{t_0,t}^\alpha f(t)$; $f\in L^p(t_0,t_1;\mathbb{R})$, for some value $p\geq2$, $g_{1-\alpha}*f\in W^{1,1}(t_0,t_1;\mathbb{R})$ and $g_{1-\alpha}*f^2\in W^{1,1}(t_0,t_1;\mathbb{R})$;
\item[(ii)] Function $f$ needs to satisfy the hypotheses of Theorem \ref{novamolli2}; more specifically, we assume that $f\in L^p(t_0,t_1;\mathbb{R})$ with $p>1/(1-\alpha)$;
\item[(iii)] The regularity that is necessary to apply Holder's inequality in \eqref{cccmainlimite} and apply the limits; more precisely $g_{1-\alpha}*f\in W^{1,2}(t_0,t_1;\mathbb{R})$.
\end{itemize}
\item[(b)] We conjecture that the restriction over $p$ and $\alpha$, which were presented in Theorem \ref{finalriemann} and discussed in item $(ii)$ above, cannot be removed; recall that Hardy-Littlewood have already struggled with this kind of restriction (see \cite[Theorem 4]{HaLi1} for details).
\end{itemize}
\end{remark}

\subsection{An improvement of inequality (\ref{ultultcpt})}

Now we address the Caputo fractional derivative. For this scenario we need a slightly distinct approach, since here the functions are at least continuous.

We begin by proving an auxiliary result, which can be considered a more regular version of Theorem \ref{finalriemann}.

\begin{lemma}\label{finalriemann2} Let $f\in C^0([t_0,t_1];\mathbb{R})$ which also satisfies $g_{1-\alpha}*f\in W^{1,1}(t_0,t_1;\mathbb{R})$ and $g_{1-\alpha}*f^2\in W^{1,1}(t_0,t_1;\mathbb{R})$. Then,
\begin{equation*}
D_{t_0,t}^\alpha\big[f(t)\big]^2\leq2\Big[D_{t_0,t}^\alpha f(t)\Big]f(t),\quad \textrm{for almost every }t\in [t_0,t_1].
\end{equation*}
\end{lemma}

\begin{proof} Assume that $t_0=0$. With the ideas and notation introduced in the proof of Theorem \ref{finalriemann}, however here with new regularities, we deduce that
\begin{equation}\label{cccc1}\left\{\begin{array}{l}F_\varepsilon\in C^\infty([0,t_1];\mathbb{R})\,\,\textrm{and}\vspace{0.2cm}\\
F_\varepsilon\rightarrow f,\textrm{ when }\varepsilon\rightarrow0,\,\,\textrm{in the topology of }C^0([\tau_0,\tau_1];\mathbb{R}).\end{array}\right.\end{equation}
for any $[\tau_0,\tau_1]\subset(0,t_1)$ and
\begin{equation}\label{cccc2}\left\{\begin{array}{l}D_{0,t}^\alpha F_\varepsilon\in C^\infty\big([0,t_1];\mathbb{R}\big)\,\,\textrm{and}\vspace{0.2cm}\\
D_{0,t}^\alpha F_\varepsilon\rightarrow D_{0,t}^\alpha f,\textrm{ when }\varepsilon\rightarrow0,\,\,\textrm{in the topology of }L^1(0,t_1;\mathbb{R}).\end{array}\right.\end{equation}

Statements \eqref{cccc1} and \eqref{cccc2} are enough for us to repeat the last part of the proof presented in Theorem \ref{finalriemann} in any interval $[\tau_0,\tau_1]\subset(0,t_1)$, what allows us to conclude the desired result. For the case $t_0\not=0$ just replicate the argument used in the last part of the proof of Theorem \ref{finalriemann}.
\end{proof}

We end this subsection with the theorem that improves inequality \eqref{ultultcpt}.

\begin{theorem}\label{finalcaputo} Assume the same hypothesis of Lemma \ref{finalriemann2}. Then,
\begin{equation*}
cD_{t_0,t}^\alpha\big[f(t)\big]^2\leq2\Big[cD_{t_0,t}^\alpha f(t)\Big]f(t),\quad \textrm{for almost every }t\in [t_0,t_1].
\end{equation*}
\end{theorem}

\begin{proof}Define function $g(t):=f(t)-f(t_0)$. Since function $g\in C^0([t_0,t_1];\mathbb{R})$ and also $g_{1-\alpha}*g\in W^{1,1}(t_0,t_1;\mathbb{R})$ and $g_{1-\alpha}*g^2\in W^{1,1}(t_0,t_1;\mathbb{R})$, by Lemma \ref{finalriemann2} it holds that
\begin{equation*}
D_{t_0,t}^\alpha\big[g(t)\big]^2\leq2\Big[D_{t_0,t}^\alpha g(t)\Big]g(t),\quad \textrm{for almost every }t\in [t_0,t_1].
\end{equation*}

On the other hand, observe that Proposition \ref{properties} ensures the identities
\begin{align*}
D_{t_0,t}^\alpha\big[g(t)\big]^2&=D_{t_0,t}^\alpha\big[f(t)\big]^2
-2D_{t_0,t}^\alpha\big[f(t)\big]f(t_0)+D_{t_0,t}^\alpha\big[f(t_0)\big]^2
\\&=D_{t_0,t}^\alpha\left\{\big[f(t)\big]^2-\big[f(t_0)\big]^2\right\}
-2D_{t_0,t}^\alpha\big[f(t)\big]f(t_0)+2D_{t_0,t}^\alpha\big[f(t_0)\big]^2
\\&=cD_{t_0,t}^\alpha\big[f(t)\big]^2
-2D_{t_0,t}^\alpha\big[f(t)\big]f(t_0)+2D_{t_0,t}^\alpha\big[f(t_0)\big]^2,
\end{align*}
and
\begin{align*}
2\Big[D_{t_0,t}^\alpha g(t)\Big]g(t)&=2D_{t_0,t}^\alpha\Big[f(t)-f(t_0)\Big]f(t)
-2D_{t_0,t}^\alpha\Big[f(t)-f(t_0)\Big]f(t_0)\\
&=2\Big[cD_{t_0,t}^\alpha f(t)\Big]f(t)-2D_{t_0,t}^\alpha\big[f(t)\big]f(t_0)+2D_{t_0,t}^\alpha\big[f(t_0)\big]^2.
\end{align*}

Thus the proof of the theorem follows.
\end{proof}

\subsection{Theorem \ref{polinomio} for vectorial functions}

The classical generalization of \eqref{desifrac} to functions with values in a Hilbert space can be summarized by the following theorem (this result is discussed in several literatures; we may cite \cite{LiMa1,Te1} as examples).
\begin{theorem}\label{originalteman} Suppose that $H$ and $V$ are Hilbert spaces satisfying:
\begin{itemize}
\item[(a)] $V$ is dense in $H$ and also is continuously included in $H$.
\item[(b)] If $H^\prime$ represents the dual of $H$, by the Riesz
representation theorem, we consider the identification  $H\equiv H^\prime$.
\end{itemize}

With the conditions introduced by items $(a)$ and $(b)$, if $V^\prime$ represents the dual of $V$, we arrive at the continuous inclusions
$$V\subset H\equiv H^\prime\subset V^\prime.$$

In this case, if
$$u\in L^2(0,T;V)\qquad\textrm{and}\qquad u^\prime\in L^2(0,T;V^\prime),$$
then $u$ is almost everywhere equal to a continuous function from $[0,T]$ into $H$ and
$$\dfrac{d}{dt}\big\|u(t)\big\|_H^2=2\left\langle\dfrac{d}{dt}u(t),u(t)\right\rangle_{V^\prime,V},\quad \textrm{for almost every }t\in [0,T].$$
Above the symbol $\langle\cdot,\cdot\rangle_{V^\prime,V}$ denotes the duality pairing.
\end{theorem}

Hence, in order to generalize Theorems \ref{finalriemann} and \ref{finalcaputo} to functions with values in a Hilbert space, we add several new ideas to the classical proof of Theorem \ref{originalteman}, so that we can overcome the barriers imposed by the non-local definition of Riemann-Liouville and Caputo fractional derivatives.

\begin{theorem}\label{firsthilber} Consider $W$ a Hilbert space and define the set
\begin{multline*}C^1([t_0,t_1];\mathbb{R})\otimes W\\=\bigcup_{n=1}^\infty\left\{\sum_{k=1}^n{\phi_k(t)w_k}:\phi_k\in C^1([t_0,t_1];\mathbb{R})\textrm{ and }w_k\in W,\textrm{ for each }k\in\{1,2,\ldots,n\}\right\}.\end{multline*}

For every function $u\in C^1([t_0,t_1];\mathbb{R})\otimes W$, it holds that:
$$D_{t_0,t}^\alpha \big\|u(t)\big\|_W^2\leq2\Big(D_{t_0,t}^\alpha u(t),u(t)\Big)_W\,\,,\quad \textrm{for every }t\in (t_0,t_1],$$
and
$$cD_{t_0,t}^\alpha \big\|u(t)\big\|_W^2\leq2\Big(cD_{t_0,t}^\alpha u(t),u(t)\Big)_W,\quad \textrm{for every }t\in [t_0,t_1].$$
\end{theorem}

\begin{proof} Recall that Zorn's lemma ensures the existence of an orthonormal Hamel basis to $W$, which we denote by $B$. Thus, we rewrite
\begin{multline*}C^1([t_0,t_1];\mathbb{R})\otimes W\\=\bigcup_{n=1}^\infty\left\{\sum_{k=1}^n{\phi_k(t)v_k}:\phi_k\in C^1([t_0,t_1];\mathbb{R})\textrm{ and }v_k\in B,\textrm{ for each }k\in\{1,2,\ldots,n\}\right\}.\end{multline*}

Therefore, since each $v\in C^1([t_0,t_1])\otimes W$ can be expressed as $v(t):=\sum_{k=1}^n{\phi_k(t)v_k}$, Proposition \ref{properties} and Theorem \ref{c1functions} guarantees that
\begin{equation*}D_{t_0,t}^\alpha \big\|u(t)\big\|_W^2=\sum_{k=1}^nD_{t_0,t}^\alpha\big[\phi_k(t)\big]^2\leq2\sum_{k=1}^n\left[D_{t_0,t}^\alpha\phi_k(t)\right]\phi_k(t)=2\Big(D_{t_0,t}^\alpha u(t), u(t)\Big)_W\,\,,\end{equation*}
for every $t\in (t_0,t_1].$  The inequality involving Caputo fractional derivative relies on similar arguments.
\end{proof}

We point out that function space $C^1\big([t_0,t_1]\big)\otimes W$ denotes a standard structure from the approximation theory which was extensively studied in several classical books from this area; see \cite{Ll1,Pro1} as standard references on this subject.

\begin{proposition}\label{densitynova} Let $W$ be a Hilbert space. Then $C^1([t_0,t_1])\otimes W$ is a dense subset of $C^1([t_0,t_1];W)$.
\end{proposition}

\begin{proof} Let $u\in C^1([t_0,t_1];W)$. Since $u^\prime\in C^0([t_0,t_1];W)$, Theorem 1.15 of Prolla's book \cite{Pro1} ensures the existence of $\{\tilde{u}_k\}_{k=1}^\infty\subset C^0([t_0,t_1])\otimes W$, where
\begin{multline*}C^0([t_0,t_1])\otimes W\\=\bigcup_{n=1}^\infty\left\{\sum_{k=1}^n{\phi_k(t)w_k}:\phi_k\in C^0([t_0,t_1];\mathbb{R})\textrm{ and }w_k\in W,\textrm{ for each }k\in\{1,2,\ldots,n\}\right\},\end{multline*}
such that
\begin{equation}\label{ultimotestenovo1}\sup_{s\in[t_0,t_1]}{\|u^\prime(s)-\tilde{u}_k(s)\|_W}\rightarrow 0,\end{equation}
when $k\rightarrow\infty$. If we assume that sequence $\{\tilde{u}_k(t)\}_{k=1}^\infty$ is given by
$$\tilde{u}_k(t):=\sum_{l=1}^{n_k}{\phi_l^k(t)w_l},\quad \textrm{for every }t\in [t_0,t_1],$$
and define $\{u_k(t)\}_{k=1}^\infty$ by
$$u_k(t):=u(t_0)+\sum_{l=1}^{n_k}{\left[\int_{t_0}^{t}{\phi_l^k(s)}\,ds\right]w_l},\quad \textrm{for every }t\in [t_0,t_1],$$
we shall achieve that $\{u_k\}_{k=1}^\infty\subset C^1([t_0,t_1])\otimes W$ and
\begin{equation}\label{ultimotestenovo2}\|u(t)-u_k(t)\|_W=\left\|\int_{t_0}^{t}{u^\prime(s)}\,ds-\int_{t_0}^{t}{\tilde{u}_k(s)}\,ds\right\|_W\leq [t_1-t_0]\sup_{s\in[t_0,t_1]}{\|u^\prime(s)-\tilde{u}_k(s)\|_W}.\end{equation}

From equations \eqref{ultimotestenovo1} and \eqref{ultimotestenovo2} we finally deduce that $u_k\rightarrow u$ in the topology of $C^1([t_0,t_1];W)$, when $k\rightarrow\infty$.
\end{proof}

An expected consequence of the above results is the following theorem.

\begin{theorem}\label{finalcorocoro} If $W$ is a Hilbert space, then for every function $u\in C^1([t_0,t_1];W)$ it holds that:
$$D_{t_0,t}^\alpha \big\|u(t)\big\|_W^2\leq2\Big(D_{t_0,t}^\alpha u(t), u(t)\Big)_W,\quad \textrm{for every }t\in (t_0,t_1],$$
and
$$cD_{t_0,t}^\alpha \big\|u(t)\big\|_W^2\leq2\Big(cD_{t_0,t}^\alpha u(t),u(t)\Big)_W,\quad \textrm{for every }t\in [t_0,t_1].$$
\end{theorem}

\begin{proof} Observe that Proposition \ref{densitynova} ensures that for any $u\in C^1([t_0,t_1];W)$ there exists a sequence $\{u_k\}_{k=1}^\infty\subset C^1([t_0,t_1])\otimes W$ such that
\begin{equation}\label{novofakeequality0}{\big\|u_k-u\big\|_{C^0([t_0,t_1];W)}}+{\left\|u_k^\prime-u^\prime\right\|_{C^0([t_0,t_1];W)}}\rightarrow 0,\end{equation}
when $k\rightarrow\infty$.

By noticing that $[t_0,t_1]\ni t\rightarrow \|u(t)\|^2_W$ and $[t_0,t_1]\ni t\rightarrow \|u_k(t)\|^2_W$ are continuously differentiable real functions with
$$\dfrac{d}{dt}\|u(t)\|^2_W=2\Big(u^\prime(t),u(t)\Big)_W\quad\textrm{and}\quad\dfrac{d}{dt}\|u_k(t)\|^2_W=2\Big(u_k^\prime(t),u_k(t)\Big)_W,$$
by Remarks \ref{caputo1} and \ref{caputo1supernovo} we deduce
\begin{multline*}\Big|D_{t_0,t}^\alpha \big\|u(t)\big\|_W^2-D_{t_0,t}^\alpha \big\|u_k(t)\big\|_W^2\Big|\\\leq\Big|cD_{t_0,t}^\alpha \big\|u(t)\big\|_W^2-cD_{t_0,t}^\alpha \big\|u_k(t)\big\|_W^2\Big|+\dfrac{(t-t_0)^{-\alpha}}{\Gamma(1-\alpha)}\Big|\big\|u(t_0)\big\|_W^2-\big\|u_k(t_0)\big\|_W^2\Big|,
\end{multline*}
for every $t\in (t_0,t_1]$, and therefore
\begin{multline*}
\Big|D_{t_0,t}^\alpha \big\|u(t)\big\|_W^2-D_{t_0,t}^\alpha \big\|u_k(t)\big\|_W^2\Big|\leq2\Big|J_{t_0,t}^{1-\alpha}\Big(u^\prime(t),u(t)\Big)_W-J_{t_0,t}^{1-\alpha}
\Big(u_k^\prime(t),u_k(t)\Big)_W\Big|\\+\dfrac{(t-t_0)^{-\alpha}}{\Gamma(1-\alpha)}\Big|\big\|u(t_0)\big\|_W^2
-\big\|u_k(t_0)\big\|_W^2\Big|,
\end{multline*}
for every $t\in (t_0,t_1]$. By observing that Cauchy-Schwarz inequality and Proposition \ref{samkokilbas} ensure
\begin{multline*}\Big|J_{t_0,t}^{1-\alpha}\Big(u^\prime(t),u(t)\Big)_W-J_{t_0,t}^{1-\alpha}
\Big(u_k^\prime(t),u_k(t)\Big)_W\Big|\leq K_1\bigg\{\big\|u^\prime-u_k^\prime\big\|_{C^0([t_0,t_1];W)}\big\|u\big\|_{C^0([t_0,t_1];W)}\\+\big\|u_k^\prime\big\|_{C^0([t_0,t_1];W)}\big\|u-u_k\big\|_{C^0([t_0,t_1];W)}\bigg\},
\end{multline*}
and
\begin{multline*}\Big|\big\|u(t_0)\big\|_W^2
-\big\|u_k(t_0)\big\|_W^2\Big|\leq \big\|u-u_k\big\|_{C^0([t_0,t_1];W)}\big\|u\big\|_{C^0([t_0,t_1];W)}\\+\big\|u_k\big\|_{C^0([t_0,t_1];W)}\big\|u-u_k\big\|_{C^0([t_0,t_1];W)},
\end{multline*}
for every $t\in (t_0,t_1]$, we deduce that
\begin{align}\label{novofakeequality1}\Big|D_{t_0,t}^\alpha \big\|u(t)\big\|_W^2-D_{t_0,t}^\alpha \big\|u_k(t)\big\|_W^2\Big|&\leq K_2\big[1+(t-t_0)^{-\alpha}\big]\bigg\{{\big\|u_k-u\big\|_{C^0([t_0,t_1];W)}}\bigg\}\\&\nonumber\hspace*{3.5cm}+K_3{\left\|u_k^\prime-u^\prime\right\|_{C^0([t_0,t_1];W)}},
\end{align}
for every $t\in (t_0,t_1]$.

On the other hand, observe
\begin{multline*}\left|\Big(D_{t_0,t}^\alpha u(t),u(t)\Big)_W-\Big(D_{t_0,t}^\alpha u_k(t),u_k(t)\Big)_W\right| \leq {\left\|D_{t_0,t}^\alpha u(t)-D_{t_0,t}^\alpha u_k(t)\right\|_W}{\left\|u(t)\right\|_W}\nonumber\\+
{\big\|D_{t_0,t}^\alpha u_k(t)\big\|_W}{\left\|u(t)-u_k(t)\right\|_W},\end{multline*}
for every $t\in (t_0,t_1]$.

Therefore, using analogous arguments, we obtain the inequality
\begin{multline}\label{novofakeequality2}\left|\Big(D_{t_0,t}^\alpha u(t),u(t)\Big)_W-\Big(D_{t_0,t}^\alpha u_k(t),u_k(t)\Big)_W\right| \\\leq  \,\,K\left\{\big[1+(t-t_0)^{-\alpha}\big]{\big\|u_k-u\big\|_{C^0([t_0,t_1];W)}}\right.+\left.{\left\|u_k^\prime-u^\prime\right\|_{C^0([t_0,t_1];W)}}\right\},\end{multline}
for every $t\in (t_0,t_1]$.

Hence, \eqref{novofakeequality0}, \eqref{novofakeequality1}, \eqref{novofakeequality2} and Theorem \ref{firsthilber} completes this proof. The arguments used to prove the second inequality are almost the same and therefore are omitted.
\end{proof}

Finally we present the full version of our main theorem, which unites practically all the results presented so far. We emphasize that this theorem is fundamental to make Faedo-Galerkin method applicable to partial differential equations with fractional time derivative.

\begin{theorem}\label{finalvectorial} Let $V$ and $H$ be Hilbert spaces that satisfies the hypothesis of Theorem \ref{originalteman}.
\begin{itemize}
\item[(a)] Assume that $p>1/(1-\alpha)$ and $p\geq2$. If $u\in L^p(t_0,t_1;V)$, $D_{t_0,t}^\alpha u\in L^2(t_0,t_1;V^\prime)$ and $g_{1-\alpha}*\|u(t)\|_H^2\in W^{1,1}(t_0,t_1;\mathbb{R})$, then $u$ is almost everywhere equal to a continuous function from $[t_0,t_1]$ into $H$ and
\begin{equation*}D_{t_0,t}^\alpha \big\|u(t)\big\|_H^2\leq2\Big\langle D_{t_0,t}^\alpha u(t),u(t)\Big\rangle_{V^\prime,V}\,\,,\quad \textrm{for almost every }t\in [t_0,t_1].\end{equation*}

\item[(b)] If $u\in L^2(t_0,t_1;V)$, $cD_{t_0,t}^\alpha u\in L^2(t_0,t_1;V^\prime)$ and $g_{1-\alpha}*\|u(t)\|_H^2\in W^{1,1}(t_0,t_1;\mathbb{R})$, then $u$ is almost everywhere equal to a continuous function from $[t_0,t_1]$ into $H$ and
\begin{equation*}cD_{t_0,t}^\alpha \big\|u(t)\big\|_H^2\leq2\Big\langle cD_{t_0,t}^\alpha u(t),u(t)\Big\rangle_{V^\prime,V}\,\,,\quad \textrm{for almost every }t\in [t_0,t_1].\end{equation*}
\end{itemize}
\end{theorem}

\begin{remark} Notice that the conditions given in Definition \ref{riemannder} are sufficient to ensure the existence of Riemann-Liouville fractional derivative, however they are not necessary. This is why we just impose in item $(a)$ the condition that $D_{t_0,t}^\alpha u$ should exists a.e. in $[t_0,t_1]$ and belong to $L^2(t_0,t_1;V^\prime)$. On the other hand, $(b)$ is more delicate. When we impose the hypotheses $cD_{t_0,t}^\alpha u\in L^2(t_0,t_1;V^\prime)$, we need to suppose that $u(t_0)$ is defined and belongs to $V^\prime$, so that Caputo fractional derivative formula can make sense. Nonetheless, this is not much. Observe that we obtain $u\in C^0([t_0,t_1];H)$ as one of the conclusions of Theorem \ref{finalvectorial}.
\end{remark}

\begin{proof}[Proof of Theorem~\ref{finalvectorial}] This proof is very similar to the one presented in Theorem \ref{finalriemann}, therefore we avoid several steps that were already implemented there.\\

 $(a)$ Assume that $t_0=0$. Let us begin by considering $\varrho\in C^\infty(\mathbb{R};\mathbb{R})$ with compact support contained in $(0,t_1)$ satisfying $\int_{-\infty}^\infty\varrho(s)ds=1$. Define $U\in L^2(\mathbb{R};V)$, by
$$U(t)=\left\{\begin{array}{ll}u(t),&\textrm{for almost every }t\in[0,t_1],\\
0,&\textrm{otherwise}.
\end{array}\right.$$
The same argument used in Theorem \ref{finalriemann} to convolutions together with Propositions \ref{samkokilbas} and \ref{novamolli1} ensure (for the classical setup of the following convergences see \cite[Lemma 1.3 of Chapter 3]{Te1} or \cite[sections 5.3.1 and 5.9.2]{Ev1})
\begin{equation}\label{convergenceresultswewewe1}\left\{\begin{array}{l}U_\varepsilon\in C^\infty([0,t_1],V)\vspace{0.2cm}\\
U_{\varepsilon}\rightarrow u,\textrm{ when }\varepsilon\rightarrow0,\,\, \textrm{in }L^p(0,t_1;V),\vspace{0.2cm}\\
g_{1-\alpha}*\|U_{\varepsilon}\|_H^2\rightarrow g_{1-\alpha}*\|u\|^2_H,\textrm{ when }\varepsilon\rightarrow0,\,\, \textrm{in }L^1(0,t_1;\mathbb{R}),\quad\textrm{and}\vspace{0.2cm}\\
D_{t}^\alpha U_{\varepsilon}\rightarrow D_{t}^\alpha u,\textrm{ when }\varepsilon\rightarrow0,\,\,\textrm{in }L^2(0,t_1;V^\prime).\end{array}\right.\end{equation}

A natural consequence of the identification made on the spaces $V$ and $H$, is that
\begin{equation}\label{finduality}\big(f,g\big)_H=\langle f,g\rangle_{V^\prime,V},\quad \textrm{for every }f\in H\textrm{ and }g\in V.\end{equation}
Thus, a direct application of Theorem \ref{finalcorocoro} and identity \eqref{finduality} to function $U_\varepsilon(t)$ gives us
\begin{equation*}D_{t}^\alpha \big\|U_\varepsilon(t)\big\|_H^2\leq2\Big(D_{t}^\alpha U_\varepsilon(t),U_\varepsilon(t)\Big)_H=2\Big\langle D_{t}^\alpha U_\varepsilon(t),U_\varepsilon(t)\Big\rangle_{V^\prime,V},
\end{equation*}
for every $t\in (t_0,t_1]$ and $\varepsilon>0$. Finally, by repeating the same arguments used in the last part of the proof of Theorem \ref{finalriemann} we achieve the desired inequality.

It remains for us to prove that $u$ is almost everywhere equal to a continuous function from $[0,t_1]$ into $H$. To this end, observe that for each subsequence $\{U_k\}_{k=1}^\infty\subset\{U_\varepsilon\}_{\varepsilon>0}$ we have
$$D_{t}^\alpha \big\|U_m(t)-U_n(t)\big\|_H^2\leq2\Big\langle D_{t}^\alpha \big[U_m(t)-U_n(t)\big],U_m(t)-U_n(t)\Big\rangle_{V^\prime,V},$$
for every $t\in(0,t_1)$ and $n,m\in\mathbb{N}^*$. By applying operator $J_{t}^\alpha$ in both sides of the inequality, Young inequality ensures
\begin{multline*}\big\|U_m(t)-U_n(t)\big\|_H^2\leq\dfrac{2}{\Gamma(\alpha)}\int_{0}^t{(t-s)^{\alpha-1}\Big\langle D_{s}^\alpha \big[U_m(s)-U_n(s)\big],U_m(s)-U_n(s)\Big\rangle_{V^\prime,V}}\,ds\\
\leq\dfrac{2}{\Gamma(\alpha)}\int_{0}^t{(t-s)^{\alpha-1}\Big[\big\|D_{s}^\alpha U_m(s)-D_{t_0,s}^\alpha U_n(s)\big\|^2_{V^\prime}+\|U_m(s)-U_n(s)\big\|^2_V\Big]}\,ds
\end{multline*}
what together with \eqref{convergenceresultswewewe1} implies that $\{U_k\}_{k=1}^\infty$ is a Cauchy sequence in $C^0([0,t_1],H)$. An usual argument now completes the proof of this item.

For the case $t_0\not=0$ we also argument like in Theorem \ref{finalriemann}.

$(b)$ The proof of this inequality follows very similar steps to those discussed in the proof of item $(a)$ together with Lemma \ref{finalriemann2} and Theorem \ref{finalcaputo}, therefore it is omitted.
\end{proof}

\section{An application of the theory in partial differential equations}
\label{applications}

In this section we apply all the techniques developed throughout this manuscript in the theory of partial differential equations with fractional derivative in the time variable.

At first we emphasize that several researchers recently started to study the incompressible Navier-Stokes equations with fractional derivative in the time variable, in the most varied circumstances; as a survey on this topic see \cite{CarPl1,ZhPe1,ZhPeAhAl1,ZoLvWu1} and references therein.

This motivated us to study the fractional version of the 2D Stokes equations in bounded domains. More specifically, we consider the following system of equations:
\begin{equation}\label{fns} \tag{FS}
\left\{\begin{array}{rl}
cD^\alpha_{0,t} u(x,t)- \rho \Delta u(x,t) + \nabla p(x)=f(x), & \quad \text{ in } \Omega\times(0,T),\\
\operatorname{div} u(x,t)=0, & \quad \text{ in } \Omega\times(0,T), \\
u(x,t)=0,&\quad\text{ on } \partial\Omega\times[0,T],\\
u(x,0)= u_0, & \quad \text{ in } \Omega.
\end{array}\right.
\end{equation}
where $\alpha\in(0,1)$ is a fixed number, $cD_t^\alpha$ is the Caputo fractional derivative of order $\alpha$ at $t_0=0$ (see Definition \ref{caputo}), $\rho$ is a positive constant, $u_0$ an initial condition, $f$ the non homogeneous term and $\Omega\subset\mathbb{R}^2$ is a bounded domain.

Consider now the function space
$$H:=\Big\{u\in \big[L^2(\Omega)\big]^2:\operatorname{div} u=0\textrm{ in }\Omega\textrm{ and }u\cdot\nu=0\textrm{ on }\partial\Omega\Big\},$$
which with the induced topology of $\big[L^2(\Omega)\big]^2$ becomes a Hilbert space. We also define the function space
$$V:=\Big\{u\in \big[H^1_0(\Omega)\big]^2:\operatorname{div} u=0\textrm{ in }\Omega\Big\},$$
that equipped with the inner product (thanks to Poincare's inequality)
\begin{equation*}\label{V}(v_1,v_2)_V:=\big(\nabla v_1,\nabla v_2\big)_H,\qquad v_1,v_2\in V,\end{equation*}
is a Hilbert space. By classical arguments, we identify $H$ with its dual space $H^\prime$ and write the inclusions
$$V\subset H\equiv H'\subset V'$$
where each space is dense in the next one and the injections are continuous.

The following definition settles down the notion of weak solution to \eqref{fns}.

\begin{definition} Consider $f\in V^\prime$ and $u_0\in H$. A weak solution to the classical problem \eqref{fns} is a function $u\in L^2(0,T;V)$ such that $cD_t^\alpha u\in L^2(0,T;V^\prime)$, $u(0)=u_0$ and $u$ satisfies the variational form
\begin{equation*}
    cD_t^\alpha (u(t),v)_H+\rho(u(t),v)_V=\big\langle f,v\big\rangle_{V^\prime,V},\quad \textrm{for almost every }t\in [0,T],
\end{equation*}
and for any $v\in V$.
\end{definition}

\begin{remark} Like in the classical setting of the Stokes equations, $u\in L^2(0,T;V)$ seems to be not enough for us to make sense of $u(0)=u_0$, since $u$ is defined almost everywhere in $[0,T]$. On the other hand, once $cD_t^\alpha u\in L^2(0,T;V^\prime)$, Theorem \ref{finalvectorial} ensures that $u\in C([0,T];H)$. This last argument eliminates any doubt regarding the condition $u(0)=u_0$.
\end{remark}

The notions introduced above are enough for us to prove the main result of this section. Just observe that in what follows we use the standard Faedo-Galerkin method mixed with arguments from Theorem \ref{finalvectorial}, and therefore we just emphasize the parts of the proof that are new.

\begin{theorem}\label{existenciaeunicidade} For given $f\in V^\prime$ and $u_0\in H$, there exists a unique weak solution to \eqref{fns}. Moreover, $u\in C([0,T];H)$.
\end{theorem}

\begin{proof} Since $V$ is an infinite dimensional separable vector space, there exists a linear independent set $\{w_1,w_2,\ldots\}$ that is total in $V$. Now for each $n\in\mathbb{N}^*$, define
$$u_n(t):=\sum_{i=1}^n{g_{in}(t)\omega_i},\quad \textrm{for every } t\in[0,T_n]\subset[0,T],$$
where $g_{in}(t)$ are the maximal local solutions of the linear differential system with Caputo fractional derivative of order $\alpha$,
\begin{equation}\label{ode}
   \left\{\begin{array}{ccc} cD_{0,t}^\alpha (u_n(t),\omega_1)_H+\rho(u_n(t),\omega_1)_V&=&\big\langle f,\omega_1\big\rangle_{V^\prime,V},\\
   cD_{0,t}^\alpha (u_n(t),\omega_2)_H+\rho(u_n(t),\omega_2)_V&=&\big\langle f,\omega_2\big\rangle_{V^\prime,V},\\
   \vdots&=&\vdots\\
   cD_{0,t}^\alpha (u_n(t),\omega_n)_H+\rho(u_n(t),\omega_n)_V&=&\big\langle f,\omega_n\big\rangle_{V^\prime,V},\end{array}\right.
\end{equation}
with initial condition $u_n(0)=u_{0n}$, where $u_{0n}$ is the orthogonal projection of $u_0$ in the subspace $[\omega_1,\omega_2,\ldots,\omega_n]$, with $u_{0n}\rightarrow u_0$ in $H$, when $n\rightarrow\infty$ (see \cite{Car1,CarFe1,AnCaCaMa1} for details on the local existence and uniqueness of solution).

Because of this, for each $n\in\mathbb{N}^*$, function $u_n\in C_\alpha([0,T_n];V)$, where
$$C_\alpha([0,T_n];V):=\Big\{u\in C([0,T_n];V):cD_t^\alpha u\in C([0,T_n];V)\Big\}.$$

Since the energy equation associated with \eqref{ode} is given by
\begin{equation*}\Big(cD_t^\alpha u_n(t),u_n(t)\Big)_H+\rho\|(u_n(t))\|_{V}^2=\big\langle f,u_n(t)\big\rangle_{V^\prime,V}\,\,,\end{equation*}
Theorem \ref{finalvectorial} allows us to obtain the inequality
\begin{equation}\label{testeineine}\frac{1}{2}cD_t^\alpha\|u_n(t)\|_{H}^2+\rho\|(u_n(t))\|_{V}^2\leq\big\langle f,u_n(t)\big\rangle_{V^\prime,V}.\end{equation}

Finally, by applying Young inequality to \eqref{testeineine} we achieve
\begin{equation}\label{energyn2}cD_t^\alpha\|u_n(t)\|_{H}^2+\rho\|(u_n(t))\|_{V}^2\leq C({\rho})\|f\|^2_{V^\prime}.\end{equation}

The bounds obtained above and the blow up result presented in \cite{Car1,AnCaCaMa1}, guarantees that the maximum time of existence $T_n=T$, for any $n\in\mathbb{N}^*$. Also, \eqref{energyn2} ensures the existence of a subsequence $\{u_{n_k}\}_{k=1}^{+\infty}$ of $\{u_{n}\}_{n=1}^{+\infty}$ and a function $u$ belonging to $L^2(0,T;V)\cap L^\infty(0,T;H)$ such that
$$\left\{\begin{array}{l}\lim_{k\rightarrow\infty}{u_{n_k}}=u, \textrm{ in the weak topology of } L^2(0,T;V)\textrm{ and }\vspace*{0.2cm}\\
\lim_{k\rightarrow\infty}{u_{n_k}}=u, \textrm{ in the weak-star topology of } L^\infty(0,T;H).\end{array}\right.$$

The above conclusions allow us to apply a limit argument in
\begin{equation*}
     (cD_t^\alpha u_{n_k}(t),v)_H+\rho(u_{n_k}(t),v)_V=\big\langle f,v\big\rangle_{V^\prime,V},
\end{equation*}
exactly as in the classical procedure, to conclude that $u$ is the weak solution of this problem. The uniqueness of solution is done with standard arguments that we avoid express here.
\end{proof}

\section{Last considerations}
\label{conclusions}

Recall that in Section \ref{pre} we have obtained two inequalities involving polynomial functions and the Leibniz rule, which in its classical formulation is recursively used to study the energy equation of PDE's. We also emphasize that the inequality with Riemann-Liouville fractional derivative is completely new to the literature of fractional calculus; this allows us to conjecture that this method could be used to prove inequalities, like the ones presented here, to other fractional derivatives.

Observe that the example presented in Section \ref{applications} is just a simple application of the inequality presented in Section \ref{maintheorem}. In fact, the Heat Equation with fractional derivatives could have be another example where this method works. We could also have used this inequalities to prove that systems of ordinary differential equations with fractional derivatives that possess a quadratic Lyapunov function (for instance $V(x,y)=x^2+y^2$) are stable or asymptotically stable (see \cite{AgOrHr1,LiChPo1,St1} for details on the definitions and previous studies) or even to study the asymptotic profile of the solutions of some partial differential equations.

At last, let us give a simple argument to support the fact that Theorem \ref{finalvectorial} can be a better option then Theorem \ref{zhou}. Our major concern resides in the fact that the solution of \eqref{ode} belongs to $C_\alpha([0,T];V)$ which is a space much bigger then $AC([0,T];V)$. This is why we cannot see clearly how to apply Theorem \ref{zhou} in the proof of Theorem \ref{existenciaeunicidade}.

\section*{Acknowledgement}
 The authors would like to thank Universidade Federal do Esp\'{\i}rito Santo and Universidade Federal de Santa Catarina for the hospitality and support during respective short term visits.

\bibliographystyle{model1-num-names}

\end{document}